%% file: hd32.tex
\documentclass{amsart}
\usepackage{amsmath, amsfonts, amssymb}
\input{makros}

\input{settings}

\begin{document}
\title[On deformation rings]{On deformation rings of residually reducible Galois representations and $R=T$ theorems}
\author{Tobias Berger$^1$ \and
Krzysztof Klosin$^2$}
\address{$^1$University of Sheffield, School of
Mathematics and Statistics, Hicks Building, Hounsfield Road, Sheffield S3
7RH, United Kingdom.}
\address{$^2$Department of Mathematics,
Queens College,
City University of New York,
65-30 Kissena Blvd,
Flushing, NY 11367, USA}                     % Do not remove



\date{10 February 2011}

\begin{abstract}
We study the crystalline universal deformation ring $R$ (and its ideal of reducibility $I$) of a mod $p$ Galois representation $\rho_0$ of
dimension $n$ whose semisimplification is the direct sum of two absolutely irreducible mutually non-isomorphic constituents $\rho_1$ and $\rho_2$. Under some assumptions on Selmer groups associated with $\rho_1$ and $\rho_2$ we show that $R/I$ is cyclic and often finite.
Using ideas and results of (but somewhat different assumptions from) Bella\"iche and Chenevier we prove that $I$ is principal for essentially self-dual representations and deduce statements about the structure of $R$. Using a new commutative algebra criterion we show that given enough information on the Hecke side one gets an $R=T$-theorem. We then apply the technique to modularity problems for 2-dimensional representations over an imaginary quadratic field and a 4-dimensional representation over $\bfQ$.
\end{abstract}

\maketitle

\section{Introduction}
Let $F$ be a number field, $\Sigma$ a finite set of primes of $F$ and $G_{\Sigma}$ the Galois group of the maximal extension
of $F$ unramified outside $\Sigma$. Let $E$ be a finite extension of $\bfQ_p$ with ring of integers $\Oo$ and residue field $\bfF$. Let $\rho_0:
G_{\Sigma} \rightarrow \GL_n(\bfF)$ be a \emph{non-semi-simple} continuous representation of the Galois group $G_{\Sigma}$ with coefficients in
$\bfF$. Suppose that $\rho_0$ has the form $$\rho_0 = \bmat \rho_1 & * \\ 0 & \rho_2\emat$$ for two absolutely irreducible continuous representations
$\rho_i: G_{\Sigma} \rightarrow \GL_{n_i}(\bfF)$ with $n_1+n_2=n$.
The goal of this article is to study the crystalline universal deformation ring
$R_{\Sigma}$ of $\rho_0$ and in favorable cases show that it is isomorphic to a Hecke algebra $\bfT_{\Sigma}$ associated to automorphic forms on some algebraic group. Our approach relies on studying the ideal of reducibility $I \subset R_{\Sigma}$ as defined by
Bella\"iche and Chenevier and the quotient $R_{\Sigma}/I$. Roughly speaking the latter ``captures''  the reducible deformations, while the former captures the irreducible ones.
As a first result we prove that under some self-duality condition imposed on the deformations the ideal $I$ is principal (section \ref{Principality}). In contrast to \cite{BellaicheChenevierbook} we do not assume that the trace of our universal deformation is ``generically irreducible''. As a result we cannot affirm that $I$ is generated by a non-zero divisor,
but this is not needed for our main results. We can, however, still show that $I$ is generated by a non-zero divisor under a certain finiteness assumption (section \ref{kerrhokerT}), which is only used for results concerning $R_{\Sigma}^{\rm red}$, the quotient of $R_{\Sigma}$ by its nilradical.

We then study the quotient $R_{\Sigma}/I$.  Under the following two major assumptions: \begin{itemize} \item that the crystalline universal deformation rings of $\rho_1$ and
$\rho_2$ are discrete valuation rings ($=\Oo$); \item that the
    Selmer group $H^1_{\Sigma}(F, \Hom_{\bfF}(\rho_2, \rho_1))$ is one-dimensional (see section \ref{section functoriality} for the definition of $H^1_{\Sigma}$),
\end{itemize}
we prove that
the $\Oo$-algebra structure map $\Oo \rightarrow R_{\Sigma}/I$ is surjective (section \ref {Cyclicity}).
 Combining this with the principality of $I$ we show (section \ref{Consequences}) that \begin{itemize} \item $R_{\Sigma}$ is a quotient of
    $\Oo[[X]]$,
\item the reduced universal deformation ring $R^{\rm red}_{\Sigma}$ is a complete intersection. \end{itemize}
The above properties give us enough control on the ring $R_{\Sigma}$ to formulate some numerical conditions (which if satisfied)
imply an $R=T$ theorem in this $n$-dimensional context (section \ref{RTtheorems}).
In fact, our method can be summarized as follows.
Suppose that we have an $\Oo$-algebra surjection $\phi: R_{\Sigma} \twoheadrightarrow \bfT_{\Sigma}$ which induces a map $\ov{\phi}: R_{\Sigma}/I \twoheadrightarrow \bfT_{\Sigma}/\phi(I)$. The surjection $\Oo \twoheadrightarrow R_{\Sigma}/I$ often factors through an isomorphism $\Oo/\varpi^m \cong R_{\Sigma}/I$. In fact the size of $R_{\Sigma}/I$ is bounded from above by the size of a certain Selmer group, $H^1_{\Sigma}(F,\Hom_{\Oo}(\tilde{\rho}_2,\tilde{\rho}_1)\otimes E/\Oo)$, where $\tilde{\rho}_i$ denotes the unique lift of $\rho_i$ to $\GL_{n_i}(\Oo)$. Thus assuming $\#H^1_{\Sigma}(F,\Hom_{\Oo}(\tilde{\rho}_2,\tilde{\rho}_1)\otimes E/\Oo)  \leq \#\bfT_{\Sigma}/\phi(I)$ we conclude that $\ov{\phi}$ is an isomorphism. We then apply a new commutative algebra criterion (section \ref{criterion}) which uses principality of $I$ as an input to allow us to conclude that $\phi$ itself must have been an isomorphism.

One way to achieve the inequality $\#H^1_{\Sigma}(F,\Hom_{\Oo}(\tilde{\rho}_2,\tilde{\rho}_1)\otimes E/\Oo)  \leq \#\bfT_{\Sigma}/\phi(I)$ is to relate both sides to the same $L$-value (these are the numerical conditions  referred to above).  Many results bounding the right-hand side from below by the relevant $L$-value are available in the literature (see sections \ref{Examples1} and \ref{Examples2} for examples of such results). A corresponding upper-bound on $\#H^1_{\Sigma}(F,\Hom_{\Oo}(\tilde{\rho}_2,\tilde{\rho}_1)\otimes E/\Oo)$ can be deduced from the relevant case of the Bloch-Kato conjecture. See Theorems \ref{cons1} and \ref{cons2}, where the numerical conditions are stated precisely. In particular it is also possible to apply our method if $R_{\Sigma}/I$ is infinite. In this case our commutative algebra criterion is an alternative to the criterion of Wiles and Lenstra.

Let us now make some remarks about relations of our approach to other modularity results. As  is perhaps obvious to the informed reader, it is different from the Taylor-Wiles method. Also, our residual representations are not ``big'' in the sense of Clozel, Harris and Taylor \cite{ClozelHarrisTaylor08}. There is some connection between our setup and that of Skinner and Wiles \cite{SkinnerWiles97}, who studied residually reducible $2$-dimesional representations of $G_{\bfQ}$, but the main arguments are different. A prototype of this method has already been employed by the authors to prove an
$R=T$-theorem for two-dimensional residually reducible Galois representations over an imaginary quadratic field \cite{BergerKlosin09, BergerKlosin11}. However, the assumptions of \cite{BergerKlosin11} are different
and the proofs reflect the ``abelian'' context of that article, and mostly could not be generalized to the current setup. In particular the principality of the ideal of reducibility in that context was a simple consequence of a certain uniqueness condition imposed on
 $\rho_0$. In the ``non-abelian'' setup the analogous uniqueness condition is almost never satisfied.  Finally let us note, that unlike recent higher dimensional modularity results of Taylor et al. (e.g. \cite{Taylor08, Geraghty10, BLGeeGeraghtyTaylor10}) which prove $R^{\rm red}=T$ theorems, our method implies that $R_{\Sigma}$ is reduced.

 In order to study crystalline deformations we establish certain functoriality properties of the Selmer groups $H^1_{\Sigma}(F, \Hom_{\Oo}(\tilde{\rho}_i, \tilde \rho_j)\otimes E/\Oo)$  for $i,j \in \{1,2\}$. In particular we need to know that these Selmer groups behave well with respect to taking fixed-order torsion elements. On the other hand our numerical criteria strongly suggest that the bounds imposed to control the order of these Selmer groups should be given by $L$-values. Conjecturally, it is the Bloch-Kato Selmer groups whose orders are controlled by these $L$-values, hence we also relate  compare these to our Selmer groups $H^1_{\Sigma}(F, \Hom_{\Oo}(\tilde{\rho}_i, \tilde \rho_j)\otimes E/\Oo)$. This is all done in section \ref{section functoriality}.

Let us now discuss the ``numerical conditions'' in more detail. They fall into two categories depending on
whether one has a reducible lift to characteristic zero or not. We will focus here on the case when no such lift exists. (We refer the reader to
Theorems \ref{cons1} and \ref{cons2} for the precise statement and to the discussion following the theorems.) One of them is a lower bound on the
order of the quotient $\bfT_{\Sigma}/\phi(I) = \bfT_{\Sigma}/J$, where $\bfT$ is a certain ``non-Eisenstein'' Hecke algebra and $J$ is the corresponding ``Eisenstein ideal''. Such
quotients (and lower bounds on them) have been studied by many authors, for example \cite{SkinnerWiles99}, \cite{B09} (where $J$ is indeed the
Eisenstein ideal) and \cite{Brown07}, \cite{Klosin09} (where $J$ is the CAP ideal - see \cite{Klosin09} for a precise definition), or
\cite{AgarwalKlosin10preprint} (where $J$ is the ``Yoshida ideal''). In general this quotient measures congruences between automorphic forms with
irreducible Galois representation and a fixed automorphic form ``lifted'' from a proper Levi subgroup (Eisenstein series, Saito-Kurowawa lift, Maass
lift, Yoshida lift). Each of these results bound this module from below by a certain $L$-value, which in fact is the $L$-value which conjecturally
gives the order of the Bloch-Kato Selmer group $H^1_{f}(F, \Hom_{\Oo}(\tilde{\rho}_2, \tilde{\rho}_1)\otimes E/\Oo)$. The second numerical condition is the upper bound on the order of a related Selmer group $H^1_{\Sigma}(F, \Hom_{\Oo}(\tilde{\rho}_2, \tilde{\rho}_1)\otimes E/\Oo)$
by the same number. This condition thus seems to require (the $\varpi$-part of) the Bloch-Kato conjecture for $\Hom_{\Oo}(\tilde{\rho}_2,
\tilde{\rho}_1)$ and is  currently out of reach in most cases when $\rho_1$ and $\rho_2$ are not characters. So, our $R=T$ result (Theorem
\ref{cons1}) should be viewed as a statement asserting that under certain assumptions on the Hecke side  (the $\varpi$-part of) the Bloch-Kato conjecture for
$\Hom_{\Oo}(\tilde{\rho}_2, \tilde{\rho}_1)$ (which in principle controls extensions of $\tilde{\rho_2}$ by $\tilde{\rho}_1$ hence \emph{reducible}
deformations of $\rho_0$) implies an $R=T$-theorem (which asserts modularity of both the reducible and the irreducible deformations of $\rho_0$). The fact that we can deduce modularity of all deformations from a statement about just reducible deformations is a consequence of the principality of the ideal of reducibility (whose size roughly speaking controls the irreducible deformations) and the commutative algebra criterion.

In the last two sections of the article we study two examples in which some (or all) of the conditions can be checked. The first example is still of an abelian nature and is in a sense a ``crystalline'' complement to our previous two articles \cite{BergerKlosin09} and
\cite{BergerKlosin11}, where we studied ordinary deformations. The second example is much less special and is in fact a prototypical
higher-dimensional problem to which we hope our result may be applied. In this example we study certain irreducible four-dimensional crystalline deformations of a representation of the form $$\rho_0 = \bmat\ov{\rho}_f(k/2-1) &* \\ 0 & \ov{\rho}_g\emat,$$ where $
\ov{\rho}_f$ and $\ov{\rho}_g$ are reductions of the Galois representations attached to two elliptic cusp forms $f$ and $g$ of weights $2$ and
$k$=even respectively. Using  results of \cite{AgarwalKlosin10preprint} and
\cite{BochererDummiganSchulzePillotpreprint} which under some assumptions provide one of the numerical conditions  (a lower bound on $\# \bfT_{\Sigma}/J$) we
prove that the Bloch-Kato conjecture in this context implies that every such deformation of $\rho_0$ is modular (i.e., comes from a Siegel modular
form). For a precise statement see Theorem \ref{modsiegel}.

 Our method of proving principality of the ideal of reducibility (section \ref{Principality}) owes a lot to the ideas of Bella\"iche and Chenevier and the authors benefited greatly from reading their book \cite{BellaicheChenevierbook}.
The authors would also like to thank the Mathematical Institute in Oberwolfach and the Max-Planck-Institut in Bonn, where part of this work was carried out for their hospitality. The first author would like to thank Queens'
College, Cambridge, and the second author would like to thank the Department of Mathematics at the University of Paris 13, where part of this work was carried out. We are also grateful to Jo\"el Bella\"iche, Ga\"etan Chenevier, Neil Dummigan, Matthew Emerton and Jacques Tilouine for helpful comments and conversations.

\section{Principality of Reducibility ideals} \label{Principality}

Let $A$ be a Noetherian henselian local (commutative) ring with maximal ideal $\fm_A$ and residue field $\bfF$ and let $R$ be an $A$-algebra. Let
$\rho: R \to M_n(A)$ be a morphism of $A$-algebras and put $T=\tr \rho: R \to A$. We assume $n!$ is invertible in $A$ and
$$\rho = \bmat \tau_1 & * \\ 0 & \tau_2\emat \mod{\fm_A}$$ is a non-semisimple extension of $\tau_2$ by $\tau_1$ for two non-isomorphic absolutely irreducible representations $\tau_i$ of dimension $n_i$.

\begin{definition} [\cite{BellaicheChenevierbook} Definition 1.5.2]
	The \emph{ideal of reducibility} of $T$ is the smallest ideal $I$ of $A$ such that $\tr(\rho)$ mod $I$ is the sum of two pseudocharacters $T_1,
T_2$ such that $T_i = \tr \tau_i$ mod $\fm_A$. We will denote it by $I_T$. (For a definition of a pseudocharacter see e.g. [loc.cit], section 1.2.) \end{definition}

\begin{definition} [\cite{BellaicheChenevierbook} Section 1.2.4]
  The kernel of a pseudocharacter $T:R \to A$ is the two-sided ideal of $R$ defined by $$\ker T= \{x \in R: \forall y \in R, T(xy)=0 \}.$$
\end{definition}

\begin{rem} \label{kerker} Note that $T$ is an $A$-module homomorphism. If $K_T$ denotes the kernel of $T$ as an
$A$-module map, then clearly $K_T \supset \ker T$. This inclusion is in general strict, and in fact it is often the case that $\ker T = \ker \rho$
(see section \ref{kerrhokerT}). \end{rem}

\begin{definition} [\cite{BellaicheChenevierbook} Definition 1.3.1]
 Let $S$ be an $A$-algebra. Then $S$ is a \emph{generalized matrix algebra (GMA) of type} $(n_1, n_2)$ if $S$ is equipped with a \emph{data of
 idempotents} $\mathcal{E}=\{e_i, \psi_i, i=1,2\}$ with

\begin{enumerate}
	\item a pair of orthogonal idempotents $e_1, e_2$ of sum 1, 	\item for each $i$, an $A$-algebra isomorphism $\psi_i: e_i S e_i \to
M_{n_i}(A)$,
\end{enumerate}
such that the trace $T: S \to A$, defined by $T(x):=\sum_{i=1}^2 \tr(\psi_i(e_i x e_i))$, satisfies $T(xy)=T(yx)$ for all $x,y \in S$.
\end{definition}

By \cite{BellaicheChenevierbook} Example 1.2.4 and Section 1.2.5 we know that both $R^{\rho}:=R/\ker \rho$ and $R^T:=R/\ker T$ are Cayley-Hamilton
quotients (cf. Definition 1.2.3 in \cite{BellaicheChenevierbook}) of $(R,T)$ and the quotient map $$\varphi: R^{\rho} \twoheadrightarrow R^T$$ is an $A$-algebra morphism with kernel $\ker T_{R^{\rho}}$ (so
a two-sided ideal).

By \cite{BellaicheChenevierbook} Lemma 1.4.3 we can now find suitable data of idempotents to give both $R^{\rho}$ and $R^T$ the structure of a GMA:

\begin{lemma}\label{data}
There exist data of idempotents $\mathcal{E}^T=\{e_i^T, \psi_i^T, i=1,2\}$ for $R^T$ and $\mathcal{E}^{\rho}=\{e_i^{\rho}, \psi_i^{\rho}, i=1,2\}$
for $R^{\rho}$ such that for $\dagger=\rho, T$

\begin{enumerate}
	\item $T(e_i^{\dagger})=n_i$, 	%\item $\tau(e_i^{\dagger})=e_i^{\dagger}$, 	
\item $\varphi(e_i^{\rho})=e_i^T$, 	\item $T(e_i x e_i)= \tr \tau_i(x) \mod{\fm_A}$, 	\item If $i \neq j$, $T(e_i x e_j y e_i) \in \fm_A$ for any $x,y \in R$, 	\item $\psi_i^{\rho} \circ \varphi=\psi_i^T$, 	\item $\psi_i^{\rho}$ lifts ${\tau_i|}_{e_i R^{\rho} e_i}: e_i R^{\rho} e_i \to M_{n_i}(\bfF)$ such that for all $x \in e_i R^{\rho} e_i$, $T(x)= \tr \psi_i^{\rho}(x)$.
\end{enumerate}
These data of idempotents define $A$-submodules $\mathcal{A}_{i,j}^{\dagger}$ of $R^{\dagger}$ for $\dagger=T, \rho$ such that there are canonical
isomorphisms of $A$-algebras $$R^{\dagger} \cong \bmat M_{n_1}(\mathcal{A}_{1,1}^{\dagger}) & M_{n_1,n_2}(\mathcal{A}_{1,2}^{\dagger}) \\
M_{n_2,n_1}(\mathcal{A}_{2,1}^{\dagger}) & M_{n_2}(\mathcal{A}_{2,2}^{\dagger}) \emat$$ and $\varphi(\mathcal{A}_{i,j}^{\rho})=\mathcal{A}_{i,j}^T$.
\end{lemma}

\begin{rem} \label{withinvo} If $R$ is endowed with an anti-automorphism $\tau$ (see below) then Lemma 1.8.3 of \cite{BellaicheChenevierbook}
ensures
that the idempotents $e_i$ as in Lemma \ref{data} can be chosen so that $\tau(e_i^{\dagger})=e_i^{\dagger}$. \end{rem}

\begin{proof} We lift the idempotents of $\ov{R}/\ker \ov{T}$ in \cite{BellaicheChenevierbook} Lemma 1.4.3 and 1.8.3 compatibly to $R^{\rho}$ and
$R^T$, i.e. such that $\varphi(e^{\rho}_i)=e^T_i$ (by first lifting them to $R^T$ and then further to $R^{\rho}$). We also choose the
$\psi_{i,j}^{\rho}$ and $\psi_{i,j}^T$ in Lemma 1.4.3 compatibly so that we can also pick $E_i^{\rho} \in e_i R^{\rho} e_i$ and $E_i^{T} \in e_i
R^{T} e_i$ (as in \cite{BellaicheChenevierbook} Notation 1.3.3) with $\varphi (E_i^{\rho})=E_i^T$.

Define $A$-submodules $\mathcal{A}_{i,j}^{*}=E_i^* R^* E_j^* \subset R^*$ for $*=\rho,T$. Note that this is how \cite{BellaicheChenevierbook} Proposition 1.4.4(i) defines $\mathcal{A}^T_{i,j}$ (i.e. via \cite{BellaicheChenevierbook} Lemma 1.4.3). By the above we then have $$\varphi(\mathcal{A}_{i,j}^{\rho})=\mathcal{A}_{i,j}^T.$$
\end{proof}

\begin{prop} \label{struct2}
One has $I_T=T(\mA_{1,2}^T \mA_{2,1}^T)$.
 \end{prop}
\begin{proof} This follows from \cite{BellaicheChenevierbook}, Proposition 1.5.1. \end{proof}

\begin{prop} \label{princ1}
 The $A$-module $\mA_{1,2}^T$ is an $A$-module generated over $A$ by one element.
\end{prop}

\begin{proof}

By \cite{BellaicheChenevierbook} Lemma 1.3.7 one can conjugate by an invertible matrix with values in $A$ (we use here that, since $A$ is local,
every finite type projective $A$-module is free) to get $\rho$ adapted to $\mathcal{E}$ in the sense of \cite{BellaicheChenevierbook} Definition
1.3.6.

Now by \cite{BellaicheChenevierbook} Proposition 1.3.8 we know that $\rho(R)$ is the standard GMA attached to some ideals $A_{1,2}^{\rho}$,
$A_{2,1}^{\rho}$ of $A$. Put $A_{1,1}^{\rho}=A_{2,2}^{\rho}=A$. The definition of adaptedness to the data of idempotents $\mathcal{E}$ means
concretely that for every $r \in R$ $$\rho(r)=\bmat a_{1,1}(r) & a_{1,2}(r)\\a_{2,1}(r) & a_{2,2}(r)\emat$$ with $a_{i,j}(r) \in
M_{n_i,n_j}(A_{i,j}^{\rho})$ and $a_{1,1}(r) \equiv \tau_1(r) \mod{\fm_A}$ and $a_{2,2}(r) \equiv \tau_2(r) \mod{\fm_A}$. Now since $\rho \otimes
\bfF$ must still be a non-split extension of $\tau_2$ by $\tau_1$ we deduce that for the image $\ov{A_{1,2}^{\rho}}$ of the ideal $A_{1,2}^{\rho}$ in
$A/\fm_A$ we have $\ov{A_{1,2}^{\rho}} \neq 0$, hence $A_{1,2}^{\rho}=A$.

By the arguments in the proof of Proposition 1.3.8 we see that we obtain the ideals $A_{i,j}^{\rho}$ of $A$ from $\mathcal{A}_{i,j}^{\rho}$ via
$A$-linear maps $f_{i,j}$ (for definition see [loc. cit]), i.e., $A_{i,j}^{\rho} = f_{i,j}(\mathcal{A}_{i,j}^{\rho})$. The maps $f_{i,j}$ are
injective
since $\rho$ is on $R^{\rho}$, hence we conclude that $\mathcal{A}_{1,2}^{\rho} \cong A$.
By Lemma \ref{data} we have $\varphi(\mathcal{A}_{1,2}^{\rho})=\mathcal{A}_{1,2}^T$. Hence $\mA^T_{1,2}$ is generated over $A$ by one element. \end{proof}

To show that $I_T$ is principal using Proposition \ref{princ1}, we will now show that $\mathcal{A}_{1,2}^T \cong \mathcal{A}_{2,1}^T$ as $A$-modules
under an additional assumption on the existence of an involution on $R$.

Let $\tau: R \rightarrow R$ be an anti-automorphism (i.e., $\tau(xy) = \tau(y)\tau(x)$) of $A$-algebras such that $\tau^2= \rm id$. For an
$A$-algebra $B$, and an $A$-algebra homomorphism $\rho: R \rightarrow M_{n}(B)$ put $\rho^{\perp} = {}^t(\rho \circ \tau)$.

\begin{example} \label{exampleaa} Here are some examples of anti-automorphisms if $R=A[G]$ for suitable Galois groups $G$:
 \begin{enumerate}
  \item $\tau: g \mapsto g^{-1}$ corresponding to $\rho^{\perp} = \rho^*$ (contragredient);
\item $\tau: g \mapsto cg^{-1}c$ for $c$ an order $2$ element in $G$ corresponding to $\rho^{\perp} = (\rho^c)^*$; \item $\tau: g \mapsto
    \chi^{-1}(g)g^{-1}$ for a character $\chi: G \rightarrow \Oo^{\times}$ corresponding to $\rho^{\perp} = \rho^*\otimes \chi^{-1}$.
 \end{enumerate}
\end{example}

Assume in addition that 	\be \label{self}T \circ \tau=T \text{ and } \tr \tau_i \circ \tau=\tr \tau_i, \quad \textup{($i=1,2$)}.\ee 	

\begin{rem}
 By \cite{BellaicheChenevierbook} p.47, if $\rho$ is a semisimple representation, valued in a field,
then $T$ being invariant under $\tau$ is equivalent to $\rho^{\perp} \cong \rho$.
\end{rem}

\begin{thm} \label{princi} If $\tau$ as in (\ref{self}) exists, then $\mathcal{A}_{1,2}^T \cong \mathcal{A}_{2,1}^T$ and $I_T$ is a principal
ideal of $A$. \end{thm}

\begin{proof} The first assertion follows from \cite{BellaicheChenevierbook}, Lemma 1.8.5(ii) (here we use that $\ker T$
is stable under involution $\tau$ which is not true for $\ker \rho$ in general). By Proposition \ref{struct2} we have
$I_T=T(\mathcal{A}_{1,2}^T\mathcal{A}_{2,1}^T)$. Let $g_{i,j}$ be a generator of $\mathcal{A}_{i,j}^T$, i.e., write $\mathcal{A}_{i,j}^T = g_{i,j}A$.
Then  $I_T=T(g_{1,2}A g_{2,1}A) = AT(g_{1,2}g_{2,1}) \subset A$. \end{proof}

\begin{rem}\label{noinvolution} If $\tau$ as in (\ref{self}) does not exist, but $\tau_1$ and $\tau_2$ are characters which satisfy
$$\dim_{\bfF}H^1(G, \Hom(\tau_1, \tau_2)) = \dim_{\bfF}H^1(G, \Hom(\tau_2, \tau_1))=1$$ then $I_T$ is principal by a result of Bella\"iche-Chenevier
and Calegari (see for example \cite{Calegari06}, Proof of Lemma 3.4). \end{rem}

\section{$\ker \rho = \ker T$} \label{kerrhokerT}
Let $\rho: R \rightarrow M_n(A)$ and $T: R \rightarrow A$ be as in the previous section.
The goal of this section is to prove Proposition \ref{kernel}. If one replaces the assumption that $A/I_T$
 be finite with the assumption that
$T\otimes K_s$ is irreducible for every $s$, then this is proved in \cite{BellaicheChenevierbook}, Proposition 1.6.4.  In this section we assume that
$A$ is reduced, write $K$ for its total fraction ring, which is a finite product of fields $K=\prod_{s \in \mS} K_s$.

\begin{prop} \label{kernel}
Assume that $A$ is reduced, infinite but $\# A/I_T<\infty$. Then $\ker \rho =\ker T$.
\end{prop}

\begin{proof}
We clearly have $\ker \rho \subset \ker T$. Put $S:=R/\ker \rho \cong \rho(R) \subset M_n(A)$. Write $T'$ for the pseudocharacter on $S$ induced by
$T$. We will show using a sequence of lemmas that $(S/\ker T') \otimes K \cong M_n(K)$. This implies that $S \otimes K \cong M_n(K)$ and therefore
that $\ker T' \otimes K=0$. Note that $\ker T'$ injects into $\ker T' \otimes K$ because the other three maps in the following (commutative) diagram
$$\xymatrix{ S \ar[r] & S\otimes K \\ \ker T' \ar[u] \ar[r] & \ker T' \otimes K\ar[u]}$$ are injective. So $\ker T'=0$ which finishes the proof of
the proposition.

To show $(S/\ker T')\otimes_A K\cong M_n(K)$ we first note that $A \hookrightarrow \prod_s A_s \subset \prod_s K_s=K$, where the products are over all minimal primes $p_s$ and $A_s=A/p_s$.

\begin{lemma} \label{inf}
$A_s$ is infinite for all $s$.
\end{lemma}

\begin{proof}
  If $A_s$ is finite then $A_s$ is a field because it is a domain. Hence $p_s$ equals the unique maximal ideal of $A$, so $p_s$ is the only minimal
  prime ideal, hence $A \subset A_s$ is a finite field, contradicting our assumption.
\end{proof}
\begin{lemma} \label{lem4}
We have $A/I_T \otimes_A K=0$ and hence $I_T \otimes_A K=K$.
\end{lemma}

\begin{proof}
  By flatness of tensoring with $K$ it suffices to show that $A/I_T \otimes_A K=0$. Denote by $\phi_s:A \twoheadrightarrow A_s$. Then $A/I_T
  \twoheadrightarrow A_s/\phi_s(I_T)$ and the latter must be finite, so by Lemma \ref{inf}  $\phi_s(I_T) \neq 0$. This implies that $A_s/\phi_s
(I_T) \otimes_{A_s}
  K_s=0$.  Now observe that $$A/I_T \otimes_A K=A/I_T \otimes_A \prod_s K_s=\prod_s A/I_T \otimes_A K_s=\prod_s A_s/\phi_s(I_T) \otimes_{A_s}
  K_s=0.$$
\end{proof}

By \cite{BellaicheChenevierbook}, Proposition 1.4.4(ii) we have $$S/\ker T' \cong \bmat M_{n_1}(A) & M_{n_1,n_2}(A_{1,2}) \\ M_{n_2, n_1}(A_{2,1}) & M_{n_2}(A)\emat \subset M_n(K),$$ for some fractional ideals (in the sense of \cite{BellaicheChenevierbook}, p.27) $A_{1,2}, A_{2,1} \subset K$.

\begin{lemma} \label{A12}
We have $A_{1,2} \otimes K= A_{2,1} \otimes K=K$.
\end{lemma}
\begin{proof}
  Let $I$ be any $K$-submodule of $K=\prod_{s \in \mS} K_s$. Then
$I=\prod_{s \in \mT \subset \mS} K_s$. By the definition of a fractional ideal (\cite{BellaicheChenevierbook} p.27) there exists $f_{i,j} \in A$ such that
$f_{i,j}A_{i,j} \subset A$ so we have $A_{i,j} \otimes K \subset K$ by flatness of $\otimes_A K$. Assume now that $A_{1,2} \otimes_A K= \prod_{s \in \mT
\subset \mS} K_s \subset K$. This implies that $A_{2,1} A_{1,2} \otimes K= \prod_{s \in \mT' \subset \mT} K_s \subset \prod_{s \in \mT \subset \mS}
K_s$ since it is a $K$-submodule.
Since $A_{i,j} \cong \mA_{i,j}$ with $\mA_{i,j}$ as in Lemma \ref{data} by \cite{BellaicheChenevierbook}, Theorem 1.4.4(ii), we have the following surjective map of $A$-modules $A_{2,1}\otimes_A A_{1,2} \cong \mA_{1,2} \otimes_A \mA_{2,1} \twoheadrightarrow \mA_{1,2}\mA_{2,1} \xrightarrow{T'} I_T$ (the last map is a surjection by Proposition \ref{struct2}),   so $A_{1,2} \otimes A_{2,1} \otimes K$ surjects onto $I_T \otimes K$ which equals $K$ by Lemma \ref{lem4}, hence we must have
$\mT'=\mT=\mS$.
\end{proof}
This finishes the proof of the proposition.
\end{proof}

\begin{cor} \label{nonzerodivisor} Assume that $A$ is reduced, inifinite but $\#A/I_T < \infty$. Also assume that $\tau$ as in (\ref{self}) exists. Then $I_T$ is principal and generated by a non-zero-divisor. \end{cor}

\begin{proof} Principality of $I_T$ follows from Theorem \ref{princi}. Arguing as in the proofs of Propositions 1.7.4 and 1.7.5 in \cite{BellaicheChenevierbook}, one sees that $I_T$ is generated by $f_{1,2}f_{2,1}$ (with $f_{i,j}$ as in the proof of Lemma \ref{A12}), which is a non-zero divisor. \end{proof}

\section{Functoriality of short crystalline representations} \label{section functoriality}

In Theorem \ref{boundons} we want to relate residual Selmer groups to Bloch-Kato Selmer groups. In this section we define these and collect some
results of Fontaine-Laffaille and Bloch-Kato on short crystalline representations and deduce a functoriality of the Selmer groups with respect to
short exact sequences of finite Galois modules. Our exposition is influenced by that of \cite{DiamondFlachGuo04} Section 2.1 and \cite{Weston00}.

\subsection{Notation for Galois cohomology}

For any field $F$, we write $G_F$ for its absolute Galois group ${\rm Gal}(\overline{F}/F)$ (for some implicit fixed choice of algebraic closure $\overline{F}$). If $F$ is a field and $M$ is a topological abelian group with an action of $G_F$, we always assume that this action is continuous with respect to the profinite topology on $G_F$ and the given topology on $M$.
If $L/K$ is an extension of fields and $M$ is a topological ${\rm Gal}(L/K)$-module, then we write $H^i(L/K,M)$ for the cohomology group $H^i({\rm Gal}(L/K),M)$, computed with continous cochains. If $L$ is a separable algebraic closure of $K$ then we just write $H^i(K,M)$.

\subsection{Local cohomology groups}

Fix a prime $p$ and let $\Oo$ be the ring of integers in a finite extension of $\bfQ_p$ and uniformizer $\varpi$. %a finite, flat, local $\bfZ_p$-algebra.
For a prime $\ell$ let $K$ be a finite extension of $\bfQ_{\ell}$.
Let $M$ be an $\Oo$-module with an $\Oo$-linear action of $G_K$. We call $M$ a $p$-adic $G_K$-module over $\Oo$ if one of the following holds:

\begin{enumerate}
  \item $M$ is \emph{finitely generated}, i.e. a finitely generated $\bfZ_p$-module and the $G_K$-action is continuous for the $p$-adic topology
      on $M$;
  \item $M$ is \emph{discrete}, i.e. a torsion $\bfZ_p$-module of finite corank (i.e. $M$ is isomorphic as a $\bfZ_p$-module to
      $(\bfQ_p/\bfZ_p)^r \oplus M'$ for some $r \geq 0$ and some $\bfZ_p$-module $M'$ of finite order) and the $G_K$-action on $M$ is continuous
      for the discrete topology on $M$;
  \item $M$ is a finite-dimensional $\bfQ_p$ vector space and the $G_K$-action is continuous for the $p$-adic topology on $M$.
\end{enumerate}
$M$ is both finitely generated and discrete if and only if it is of finite cardinality.

\begin{definition} \label{fss}
 A local finite-singular structure on $M$ consists of a choice of $\Oo$-submodule $N(K,M) \subseteq H^1(K,M)$.
\end{definition}

\subsubsection{$\ell=p$}
Consider first $\ell=p$. Assume that $K$ is unramified over $\bfQ_p$. We will be using the \emph{crystalline} local finite-singular structure,
defined in the following.

 Let $T \subseteq V$ be a $G_K$-stable $\bfZ_p$-lattice and put $W=V/T$. For $n \geq 1$, put $$W_n=\{x \in W: \varpi^n x =0\}\cong T/\varpi^n T.$$

Following Bloch and Kato we define
$N(K,V)=H^1_f(K,V)={\rm ker}(H^1(K,V)\to H^1(K,B_{\rm crys} \otimes V))$, denote by $H^1_f(K,T)$ its pullback via the natural map $T \hookrightarrow V$ and let $N(K,W)=H^1_f(K,W)={\rm im}(H^1_f(K,V)\to H^1(K,W))$.

For finitely generated $p$-adic $G_K$-modules we recall the theory of Fontaine-Laffaille \cite{FontaineLaffaille82}, following the exposition in
\cite{ClozelHarrisTaylor08} Section 2.4.1. Let $\mathcal{M F}_{\Oo}$ (``Dieudonn\'e modules'') denote the
category of finitely generated $\Oo$-modules  $M$ together with a decreasing filtration ${\rm Fil}^i M$ by $\Oo$-submodules which are $\Oo$-direct summands with
${\rm Fil}^0 M=M$ and ${\rm Fil}^{p-1} M=(0)$ and Frobenius linear maps $\Phi^i:{\rm Fil}^i M \to M$ with $\Phi^i|_{{\rm Fil}^{i+1} M} = p
\Phi^{i+1}$ and $\sum \Phi^i {\rm Fil}^i M=M$. They define an exact, fully faithful covariant functor $\mathbf{G}$
of $\Oo$-linear categories from
$\mathcal{M F}_{\Oo}$ (in their notation $\mathbb{G}_{\tilde v}$ and $\mathcal{M F}_{\Oo,\tilde v}$)  to the category of finitely generated $\Oo$-modules with
continuous action by $G_K$. Its essential image is closed under taking subquotients and contains quotients of lattices in
 short crystalline representations defined as
follows: We call $V$ a continuous finite-dimensional $G_K$-representation over $\bfQ_p$  \emph{short crystalline} if, for all places $v \mid p$,
${\rm Fil}^0 D=D$ and ${\rm Fil}^{p-1} D=(0)$ for the filtered vector space $D=(B_{\rm crys} \otimes_{\bfQ_p} V)^{G_v}$ defined by Fontaine.
 Note that this differs slightly from the definition in Section 1.1.2 of \cite{DiamondFlachGuo04} and follows instead the more restrictive setting
 of
 \cite{ClozelHarrisTaylor08} Section 2.4.1.

For any $p$-adic $G_K$-module $M$ of finite cardinality in the essential image of $\mathbf{G}$ we define $H^1_f(K,M)$ as the image of ${\rm Ext}^1_{\mathcal{M
F}_\Oo}(1_{\rm FD}, D)$ in $H^1(K,M)\cong {\rm Ext}^1_{\Oo[G_K]}(1,M)$, where $\mathbf{G}(D)=M$ and $1_{\rm FD}$ is the unit filtered Dieudonn\'e
module defined in Lemma 4.4 of \cite{BlochKato90}.

\begin{rem} \label{DFG22}
Note that we define $H^1_f(K,W)$ and $H^1_f(K, W_n)$ in two different ways (using the Bloch-Kato definition for the first group and the $\bfG$-functor for the latter). However, it is in fact true that the isomorphism $W=\dirlim_n W_n$ induces an isomorphism $H^1_f(K,W) = \dirlim_n H^1_f(K, W_n)$ (cf. Proposition 2.2 in \cite{DiamondFlachGuo04}).
\end{rem}

\begin{lemma} \label{diag1}
Let $$0 \to T' \overset{i}{\to} T \overset{j}{\to} T'' \to 0$$ be an exact sequence of finite $p$-adic $G_K$-modules in the essential image of
$\mathbf{G}$. Then there is an exact sequence of $\Oo$-modules
 $$0 \to H^0(K,T') \to H^0(K,T) \to H^0(K,T'') \to H^1_{f}(K, T') \to H^1_f(K, T) \to H^1_f(K, T'') \to 0.$$
\end{lemma}

\begin{proof}
  Let $D^*$ be elements of $\mathcal{M F}_\Oo$ such that $\mathbf{G}(D^*)=T^*$. This follows from the functoriality of the ${\rm Ext}$-functor and
  ${\rm Ext}^0(1,D)=H^0(K,\mathbf{G}(D))$ and ${\rm Ext}^2(1,D)=0$ for any Dieudonn\'e module $D$.
\end{proof}

By Lemma \ref{diag1} we have the following commutative diagram with exact rows:
$$\xymatrix{ 0 \ar[r] &H^0(K,T'')/j_*H^0(K,T) \ar[r] \ar[d]^=  &H^1_f(K, T') \ar[r] \ar[d]^{\subseteq} &H^1_f(K, T) \ar[r]
\ar[d]^{\subseteq} & H^1_f(K, T'') \ar[r] \ar[d]^{\subseteq} & 0\\0 \ar[r] & H^0(K,T'')/j_*H^0(K,T) \ar[r]   & H^1(K, T') \ar[r]^{i_*}  &H^1(K,
T)\ar[r]^{j_*} & H^1(K, T'') & }$$ This implies $$H^1_f(K, T'') = j_* H^1_f(K, T)$$ and  $$H^1_f(K, T') = i_*^{-1} H^1_f(K, T),$$ by comparing the
first row with the exact sequence $$0 \to H^0(K,T'')/j_*H^0(K,T) \to i_*^{-1} H^1_f(K, T) \to H^1_f(K, T) \to j_* H^1_f(K, T) \to 0$$ of
\cite{Weston00} Lemma I.3.1.

In the terminology of \cite{Weston00} this says that the local finite-singular crystalline structures on $T'$ and $T''$ are the \emph{induced} structures
giving the crystalline finite-singular structure on $T$.

\begin{cor}\label{induced}
  Let $W$ and $W_n$ be as above. Then we have an exact sequence of $\Oo$-modules
  $$0 \to H^0(K,W)/\varpi^n \to H^1_f(K, W_n) \to H^1_f(K,W)[\varpi^n] \to 0.$$
\end{cor}

\begin{proof}
  We apply Lemma \ref{diag1} to the exact sequence $$0 \to W_n \to W_m \xrightarrow{\cdot \varpi^n} W_{m-n} \to 0$$ for $m\geq n$. This implies the exactness of $$0 \to
  H^0(K,W_{m-n})/\varpi^n H^0(K, W_m) \to H^1_f(K, W_n) \to  H^1_f(K,W_m)[\varpi^n] \to 0.$$
By taking $\dirlim_m$ we get a short exact sequence
$$0 \to
  H^0(K,W)/\varpi^n \to H^1_f(K, W_n) \to (\dirlim_m H^1_f(K,W_m))[\varpi^n] \to 0,$$ so we conclude by Remark \ref{DFG22}.
\end{proof}

\subsubsection{$\ell \neq p$}
For primes $\ell \neq p$ we define the \emph{unramified} local finite-singular structure on any $p$-adic $G_K$-module $M$ over $\Oo$  as $$N(K,M)=H^1_{\rm
ur}(K,M)={\rm ker}(H^1(K,M) \to H^1(K_{\rm ur},M)),$$  where $K_{\rm ur}$ is the maximal unramified extension of $K$.

For an exact sequence $0 \to M' \overset{i}{\to} M \overset{j}{\to} M'' \to 0$ of unramified $p$-adic $G_K$-modules over $\Oo$ \cite{Weston00} Lemma
I.2.1 shows that this structure on $M$ induces the unramified
structures on $M'$ and $M''$, i.e. $$H^1_{\rm ur} (K, M'') = j_* H^1_{\rm ur}(K, M)$$ and  \begin{equation} \label{indunr} H^1_{\rm ur}(K, M') =
i_*^{-1} H^1_{\rm ur}(K, M).\end{equation}

Let $V$ be a continuous finite-dimensional $G_K$-representation over $\bfQ_p$ and $T \subseteq V$ be a $G_K$-stable $\bfZ_p$-lattice and put $W=V/T$.
Bloch-Kato then define the following finite-singular structures on $V$, $T$and $W$:  $$H^1_f(K,V)=H^1_{\rm ur}(K,V), $$ $$H^1_f(K,T)=i^{-1} H^1_f(K,V) \text{ for } T \overset{i}{\hookrightarrow} V$$ and $$H^1_f(K,W)={\rm im}(H^1_f(K,V) \to H^1(K,W)).$$

By \cite{Rubin00} Lemma 1.3.5 we have
$H^1_f(K,W)=H^1_{\rm ur}(K,W)_{\rm div}$.
Following \cite{Rubin00} Definition 1.3.4 we define $H^1_f(K,W_n)$ just as the inverse image of
$H^1_f(K,W)$ under the map $H^1(K,W_n) \to H^1(K,W)$.
Call this the \emph{minimally ramified structure}. For the minimally ramified structure it follows (see e.g.
\cite{Rubin00} Corollary 1.3.10) that $\dirlim_m H^1_f(K,W_m)=H^1_f(K,W)$.
Note that by \cite{Rubin00} Lemma 1.3.5(iv) the minimally ramified structure agrees with the unramified structure (i.e. $H^1_f(K,W)=H^1_{\rm ur}(K,W)$ and $H^1_f(K,W_n)=H^1_{\rm ur}(K,W_n)$) if $W$ is unramified.

\subsection{Global Selmer groups} \label{Global}
Let $F$ be a number field and let $\Sigma$ be a fixed finite set of finite places of $F$ containing the places $\Sigma_p$ lying over $p$.  Assume that $p$ is unramified in $F/\bfQ$. For every
place $v$ we fix embeddings of $\overline{F} \hookrightarrow \ov{F}_v$. We write $F_{\Sigma}$ for the maximal (Galois)
extension of $F$ unramified outside $\Sigma$ and all the archimedean places
and set $G_{\Sigma} = \Gal(F_{\Sigma}/F)$.

We use the terminology of $p$-adic finitely generated (or discrete) $G_{\Sigma}$-modules similar to the corresponding local notions.

For any $p$-adic $G_{\Sigma}$-module $M$ we defined the crystalline local finite-singular structure $H^1_f(F_v, M)$ for $v \mid p$.

\begin{definition}
  We define the Selmer group $H^1_{\Sigma}(F,M)$ of $M$ as the kernel of
  the map
  $$H^1(F_{\Sigma},M) \to \prod_{v \in \Sigma_p} H^1(F_v,M)/H^1_f(F_v,M).$$
\end{definition}
\noindent Note that this Selmer group does not impose any conditions at places in $\Sigma \backslash \Sigma_p$.

Let $V$ be a continuous finite-dimensional representation of $G_{\Sigma}$ over $\bfQ_p$ which is short crystalline. Let $T \subseteq V$ be a
$G_{\Sigma}$-stable lattice and put $W=V/T$ and $W_n$ as before.

For $v \nmid p$ let $H^1_f(F_v,M)$ denote the minimally ramified structure on $M=W, W_n$, as defined above.
We will also require the definition of the Bloch-Kato Selmer
group, which has more restrictive local conditions:

	\be \label{BKsel} H^1_f(F,W)={\rm ker}(H^1(F_{\Sigma},W) \to \prod_{v \in \Sigma} H^1(F_v,W)/H^1_f(F_v,W),\ee where $H^1_f(F_v,W)=0$ for $v \mid \infty$.

This Bloch-Kato Selmer group is conjecturally related to special $L$-values. The two groups $H^1_{\Sigma}(F,W)$ and $H^1_f(F,W)$ coincide if the latter also has no local conditions at $v \in \Sigma \backslash \Sigma_p$, i.e. when $H^1_f(K_v,W)=H^1(K_v,W)$. The following Lemma will be useful to identify such situations:

Put $V^*=\Hom_{\Oo}(V,E(1))$,
 $T^*=\Hom_{\Oo}(T,\Oo(1))$ and $W^*=V^*/T^*$. We define the $v$-Euler factor \begin{equation} P_v(V^*,X)={\rm det}(1-X{\rm  Frob}_v|_{(V^*)^{I_v}}).\end{equation}

\begin{lemma} \label{Tamagawa}
  $H^1_{\Sigma}(F,W)=H^1_f(F,W)$ if for all places $v \in \Sigma$, $v \nmid p$ we have
    \begin{enumerate}
      \item $P_v(V^*,1)\in \Oo^*$
      \item  ${\rm Tam}_v^0(T^*)=1$.
    \end{enumerate}

Here the \emph{Tamagawa factor} ${\rm Tam}_v^0(T^*)$ equals $\# H^1(F_v,T^*)_{\rm tor} \times |P_v(V^*,1)|_p$ (see \cite{Fontaine92}, Section 11.5). \end{lemma}

\begin{proof}
 Consider a finite place $v \in \Sigma$.
By \cite{Rubin00} Proposition 1.4.3 (i) we see that $H^1(F_v,W)/H^1_f(F_v,W)$ is isomorphic to $H^1_f(F_v,T^*)$.

Since the Euler factor $P_v(V^*,1) \neq 0$ we have that $H^0(F_v,V^*)=0=H^1_f(F_v,V^*)$ and so $H^1_f(F_v,T^*)=H^1(F_v,T^*)_{\rm tor}$ (see Fontaine, Asterisque 206, 1992, Section 11.5).

To conclude the lemma we note that $H^1_f(F,W)$ has additional local conditions at infinity compared to the definition of $H^1_{\Sigma}(F,W)$. However, for an archimedean place $v$ we get that $H^1(\bfR,W)=0$ since $\Gal(\bfC/\bfR)$ has order 2 and $W$ is pro-$p$, and our assumption that $p>2$.
\end{proof}

\begin{rem}  \label{tamagawaremark} We remark that the triviality of $H^0(F_v,V^*)$ and $H^1_f(F_v,V^*)$ imply via the long exact sequence associated to $0 \to T \to V \to W \to 0$ that $$H^1_f(F_v,T^*) \cong H^0(F_v,W^*).$$
In $H^0(F_v,W^*)$ one has a subgroup $((V^*)^{I_v}/(T^*)^{I_v})^{\Frob_v=1}$, which has order $|P_v(V^*,1)|_{\varpi}^{-1}$
In fact, the long exact $I_v$-cohomology sequence $$0\to (T^*)^{I_v} \to (V^*)^{I_v} \to (W^*)^{I_v} \to H^1(I_v,T^*) \to H^1(I_v,V^*)$$ tells us that the index of $((V^*)^{I_v}/(T^*)^{I_v})^{\Frob_v=1}$ in $H^0(F_v,W^*)$ is given by $\#(H^1(I_v,T^*)_{\rm tor}^{G_v})$. By Proposition 4.2.2 in \cite{FontainePerrin-Riou94} we know that the latter equals ${\rm Tam}_v^0(T^*)$. This implies that ${\rm Tam}_v^0(T^*)$ is trivial if $W^{I_v}$ is divisible.
\end{rem}

\begin{prop} \label{functoriality} If  $H^0(F_{\Sigma},W)=0$ then we have
  $$H^1_{\Sigma}(F,W_n)\cong H^1_{\Sigma}(F,W)[\varpi^n].$$

\end{prop}

\begin{proof}
  We note that the local finite-singular structures on $W_n$ are induced from those on $W$ under the natural inclusion $W_n \hookrightarrow W$ (by
  (\ref{indunr}) for $v \nmid p$ or by Corollary \ref{induced} and the discussion preceding it for $v \mid p$). Using this, one shows by a
  diagram chase  (see proof of \cite{Weston00} Lemma II.3.1) that the exact sequence $$ 0 \to W_n \to W \overset{\times
  \varpi^n}{\to} W \to 0$$ gives rise to an exact sequence  $$0 \to H^0(F_{\Sigma},W)/\varpi^n \to H^1_{\Sigma}(F,W_n) \to H^1_{\Sigma}(F,W)[\varpi^n] \to
  0.$$
\end{proof}

To conclude this section, we define the notion of a \emph{crystalline} representation, following \cite{ClozelHarrisTaylor08} p. 35. Let $v \mid p$
and $A$ be a complete Noetherian $\bfZ_p$-algebra. A representation $\rho: G_{F_v} \to {\rm GL}_n(A)$ is crystalline if for each Artinian quotient
$A'$ of $A$, $\rho \otimes A'$ lies in the essential image of $\mathbf{G}$.

\section{Setup for universal deformation ring} \label{Setup}

\subsection{Main assumptions} \label{Main assumptions}

Let $F$ be a number field and $p>2$ a prime with $p \nmid \# \Cl_F$  and $p$ unramified in $F/\bfQ$. Let $\Sigma$ be a finite set of finite places of $F$ containing all the places
lying over $p$. Let $G_{\Sigma}$ denote the Galois
group $\Gal(F_{\Sigma}/F)$, where $F_{\Sigma}$ is the maximal extension of $F$ unramified outside $\Sigma$.  For every prime $\fq$ of $F$ we fix compatible embeddings $\ov{F} \hookrightarrow \ov{F}_{\fq} \hookrightarrow \bfC$ and write $D_{\fq}$ and $I_{\fq}$ for the corresponding decomposition and inertia subgroups of $G_F$ (and also their images in $G_{\Sigma}$ by a slight abuse of notation). Let $E$ be a (sufficiently large) finite extension of $\bfQ_p$ with ring of integers $\Oo$ and residue field $\bfF$. We fix a choice
of a  uniformizer $\varpi$. Consider the following $n$-dimensional residual representation: $$\rho_0=\bmat \rho_1 & * \\ & \rho_2\emat: G_{\Sigma} \rightarrow \GL_n(\bfF).$$ We assume that $\rho_1$ and $\rho_2$ are absolutely irreducible and non-isomorphic (of dimensions $n_1, n_2$ respectively with $n_1+n_2=n$)
and that $\rho_0$ is non-semisimple. From
now on assume $p \nmid n!$. Furthermore, we assume that $\rho_0$ is crystalline at the primes of $F$ lying over $p$.

For $i=1,2$ let $R_{i, \Sigma}$ denote the universal deformation ring (so in particular a local complete Noetherian $\Oo$-algebra with residue field
$\bfF$) classifying all $G_{\Sigma}$-deformations of $\rho_i$ that are crystalline at the primes dividing $p$. So, in particular we do not impose on
our lifts any conditions at primes in $\Sigma \setminus \Sigma_p$.
\begin{assumption} \label{mainass} In what follows we make the following assumptions:
\begin{enumerate}
\item  $\dim_{\bfF}H^1_{\Sigma}(F, \Hom_{\bfF}(\rho_2, \rho_1))=1.$ \item  $R_{1, \Sigma} = R_{2, \Sigma}=\Oo$. Set $\tilde{\rho}_i$, $i=1,2$ to
    be the unique deformations of $\rho_i$ to $\GL_{n_i}(\Oo)$.
\end{enumerate}
\end{assumption}

Note that Assumptions \ref{mainass} put certain restrictions on the ramification properties of the representations $\rho_i$. Set $V_{i,j}:= \Hom_{\Oo}(\tilde{\rho}_i, \tilde{\rho}_j)\otimes E/\Oo$ for $i,j \in \{1,2\}$. Fix a $G_{\Sigma}$-stable $\Oo$-lattice $T_{i,j}$ in $V_{i,j}$ and write $W_{i,j}=V_{i,j}/T_{i,j}$. Assumption \ref{mainass}(2) is
equivalent to the following two assertions:
\begin{itemize}
\item $H^1_{\Sigma}(F, W_{i,i}[\varpi])=0$ for $i=1,2$. \item There exists a crystalline lift of $\rho_i$ to $\GL_{n_i}(\Oo)$.
\end{itemize}
So, apart from the existence of the lift, both conditions (1) and (2) can be viewed as conditions on some Selmer groups, more specifically $H^1_{\Sigma}(F, W_{i,i}[\varpi])$ and $H^1_{\Sigma}(F, W_{2,1}[\varpi]) = H^1_{\Sigma}(F,\Hom_{\bfF}(\rho_2, \rho_1))$. When $\Sigma$ consists only of the primes of $F$ lying
above $p$, then $H^1_{\Sigma}(F, W_{i,j}) = H^1_f(F,W_{i,j})$ and the size of the latter group is (conjecturally) controlled by an (appropriately normalized) $L$-value $L_{i,j}$. By Proposition \ref{functoriality} we have $H^1_f(F,W_{i,j}[\varpi]) = H^1_f(F,W_{i,j})[\varpi]$. In particular if $\Sigma=\Sigma_p$ and $L_{i,i}$ is a $p$-adic unit and $L_{2,1}$ has $\varpi$-adic valuation equal to 1, the conditions on the Selmer groups are satisfied. (A weaker condition guaranteeing cyclicity of $H^1_f(F, W_{2,1})$ would suffice, but cannot be read off from an $L$-value.)
However, in the situations when $\Sigma \neq \Sigma_p$, the Selmer groups $H^1_{\Sigma}$ are not necessarily the same as the Bloch-Kato Selmer groups $H^1_f$.
For all the applications that we have in mind the following assumption on the set $\Sigma$ allows us to control the orders of Selmer groups involved
in the arguments:

Assume that for all places $v \in \Sigma$, $v \nmid p$ and all pairs $(i,j)\in \{(1,1), (2,2), (2,1)\}$ we have
    \begin{enumerate}
      \item $P_v((V_{i,j})^*,1)\in \Oo^*$
      \item  ${\rm Tam}_v^0((T_{i,j})^*)=1$.
    \end{enumerate}

By Lemma \ref{Tamagawa} we then know that we have $H^1_{\Sigma}(F,W_{i,j})=H^1_f(F,W_{i,j})$, so in this case the $L$-value conditions discussed above suffice. Also note that in the case $i=j$, $W_{i,i}=\ad \tilde \rho_i = \ad^0 \tilde \rho_i \oplus \bfF$, so the condition reduces to a condition on $H^1_{\Sigma}(F, \ad^0 \tilde\rho_i \otimes E/\Oo)$ as long as we assume that $\Sigma$ does not contain any prime $v$ with $\# k_v \equiv 1$ mod $p$ because then the condition $p \nmid \#\Cl_F$ ensures that $H^1_{\Sigma}(F, \bfF)=0$.

\subsection{Definitions} \label{Deformations of rho_0}

From now on we assume
that the representations $\rho_1$ and $\rho_2$ as well as the set $\Sigma$ satisfy Assumption \ref{mainass} and that $\rho_0$ is crystalline. Denote
the category of local complete Noetherian $\Oo$-algebras with residue field $\bfF$ by $\textup{LCN}(E)$. An $\Oo$-deformation of $\rho_0$ is a pair
consisting of $A \in \textup{LCN}(E)$ and a strict equivalence class of continuous representations $\rho: G_{\Sigma} \rightarrow \GL_{n}(A)$
such that $\rho_0 = \rho \pmod{\fm_A}$, where $\fm_A$ is the maximal ideal of $A$. As is customary we will denote a deformation by a single member of
its strict equivalence class.

\begin{definition} \label{sigmamin} We say that an $\Oo$-deformation
$\rho: G_{\Sigma} \rightarrow \GL_{n}(A)$ of $\rho_0$ is \emph{crystalline}  if  $\rho|_{D_{\fq}}$ is crystalline at
the primes $\fq$ lying over $p$.
\end{definition}

\begin{lemma} The representation $\rho_0$ has scalar centralizer. \end{lemma}
\begin{proof}
Let $\bmat A & B \\ C & D \emat \in \GL_{n}(\bfF)$ lie in the centralizer of $\rho_0$, i.e., $$ \bmat A & B \\ C & D \emat \bmat \rho_1 & f \\
0 & \rho_2\emat = \bmat \rho_1 & f \\ 0 & \rho_2\emat \bmat A & B
\\ C & D \emat,$$ where all the matrices are assumed to have appropriate
sizes. Then $C \rho_1 = \rho_2 C$, hence $C=0$, because $\rho_1 \not\cong \rho_2$. This forces $A$ (resp. $D$) to lie in the centralizer of $\rho_1$
(resp. $\rho_2$), hence $A$ and $D$ are scalar matrices (equal to, say, $a$ and $d$ respectively) by Schur's lemma, since $\rho_1$ and $\rho_2$ are
absolutely irreducible. Now, since $\rho_0$ is not split, there exists $g \in G_{\Sigma}$ such that $\rho_1(g)=I_{n_1}$ and $\rho_2(g)=I_{n_2}$
(identity matrices), but $f(g)\neq 0$. Then the identity $$af + B\rho_2 = \rho_1 B + fd$$ implies that $a=d$, hence it reduces to $B\rho_2 = \rho_1
B$, which implies that $B=0$ since $\rho_1 \not\cong \rho_2$.
\end{proof}
Since $\rho_0$ has a scalar centralizer and crystallinity is a deformation condition in the sense of \cite{Mazur97}, there exists a universal
deformation ring which we will denote by $R'_{\Sigma} \in \textup{LCN}(E)$, and a universal crystalline $\Oo$-deformation $\rho'_{\Sigma} :
G_{\Sigma} \rightarrow \GL_{n}(R'_{\Sigma})$  such that for every $A \in \textup{LCN}(E)$ there is a one-to-one correspondence between the set
of $\Oo$-algebra maps $R'_{\Sigma} \rightarrow A$ (inducing identity on $\bfF$) and the set of crystalline deformations $\rho: G_{\Sigma}
\rightarrow \GL_{n}(A)$ of $\rho_0$.

Suppose that there exists an anti-automorphism $\tau$ as in (\ref{self}).

\begin{definition} \label{sd} For $A \in \textup{LCN}(E)$ we call a
crystalline deformation $\rho: G_{\Sigma} \rightarrow \GL_{n}(A)$ \emph{$\tau$-self-dual} or simply \emph{self-dual} if $\tau$ is clear
from the context if $$\tr \rho = \tr \rho \circ \tau.$$ \end{definition}

\begin{prop} \label{reprsd} The functor assigning to an object $A \in
\textup{LCN}(E)$ the set of strict equivalence classes of self-dual crystalline deformations to $\GL_{n}(A)$ is representable by the quotient of
$R'_{\Sigma}$ by the ideal generated by $\{\tr \rho_{\Sigma}(g)-\tr \rho_{\Sigma}(\tau(g)) \mid g \in G_{\Sigma} \}$. We will denote this quotient by
$R_{\Sigma}$ and will write $\rho_{\Sigma}$ for the corresponding universal deformation. \end{prop}

We write $R^{\rm red}_{\Sigma}$ for the quotient of $R_{\Sigma}$ by its nilradical and $\rho_{\Sigma}^{\rm red}$ for the corresponding (universal)
deformation, i.e., the composite of $\rho_{\Sigma}$ with $R_{\Sigma}\twoheadrightarrow R_{\Sigma}^{\rm red}$.
 We will also write $I_{\rm re} \subset R_{\Sigma}$ for the ideal of reducibility of $\tr \rho_{\Sigma}$ and $I'_{\rm re} \subset R'_{\Sigma}$ for the ideal of
reducibility of $\tr \rho'_{\Sigma}$, and finally $I_{\rm re}^{\rm red}$ for the ideal of reducibility of $\tr \rho_{\Sigma}^{\rm red}$.
The results of Section 1 tell us:
\begin{prop} \label{princi2} The ideal of reducibility $I_{\rm re}\subset R_{\Sigma}$ (resp. $I_{\rm re}^{\rm red}
\subset R_{\Sigma}^{\rm red}$) of $\tr \rho_{\Sigma}$ (resp. $\tr \rho^{\rm red}_{\Sigma}$) is principal. \end{prop}

\section{Upper-triangular deformations of $\rho_0$} \label{Reducible
deformations}

In this section we study deformations of $\rho_0$ to complete local rings whose trace splits as a sum of two pseudocharacters.

\subsection{No infinitesimal upper-triangular deformations} \label{s3.1}

\begin{definition} We will say that a crystalline deformation is \emph{upper-triangular} if some member of its strict equivalence class has the
form $$\rho(g)=\bmat A_1(g) & B(g) \\ 0 & A_2(g)\emat \quad \textup{for all $g\in G_{\Sigma}$}$$ with $A_i(g)$ an $n_i \times n_i$-matrix.
\end{definition}

\begin{prop} \label{nored} Under Assumption \ref{mainass} (1) and
(2) there does not exist any non-trivial upper-triangular crystalline deformation of $\rho_0$ to $\GL_{n}(\bfF[x]/x^2)$. \end{prop}

\begin{proof} Let $\rho'=\bmat \rho'_1 & * \\ & \rho'_2\emat$ be such a
deformation. By Assumption \ref{mainass} (2), we have  that $\rho'_i$ is strictly equivalent to $\rho_i$ for $i=1,2$. By conjugating it by an
upper-block-diagonal matrix with entries in $\bfF$ and identity matrices in the blocks on the diagonal we may assume that $\rho'_i=\rho_i$. Assume
$*=f+xg$. In the basis $$\bmat 1 \\ 0 \emat, \quad \bmat 0 \\ 1 \emat, \quad \bmat x \\ 0 \emat, \quad \bmat 0 \\ x \emat,$$ the representation
$\rho'$ has the following form
$$\rho'= \bmat \rho_1 & f \\ & \rho_2 \\ & g & \rho_1 & f \\ &&& \rho_2
\emat.$$ Hence it has a subquotient isomorphic to $$ \tau:=\bmat \rho_1 & g \\ & \rho_2\emat.$$ Note that $\tau$ as a subquotient of a crystalline
representation is still crystalline, thus $g$ gives rise to an element in $H^1_{\Sigma}(F, \Hom_{\bfF}(\rho_2, \rho_1))$. If $g$ is the trivial
class, then we get $\rho' \cong \rho_0$ as claimed, so assume that $g$ is non-trivial. Then we must have $\tau \cong \rho_0$ by Assumption
\ref{mainass}(1). Hence there exists $Y:=\bsmat A&B\\ C&D \esmat \in \GL_2(\bfF)$ such that $Y\rho_0 = \tau Y$. Using the fact that $\rho_1, \rho_2$
are irreducible and non-isomorphic an easy calculation shows that $a=A$, $d=D$ must be scalars, $C=0$ and that
$$g=d^{-1}
(af + B\rho_2 - \rho_1 B).$$ Set $$Z=\bmat 1& -d^{-1} B x \\ & 1+ \frac{a}{d} x\emat \in \GL_2(\bfF[x]/x^2).$$ Then one checks easily that
$$Z\rho' = \rho_0 Z,$$ hence we are done.
\end{proof}

\subsection{Study of upper-triangular deformations to cyclic $\Oo$-modules} \label{s3.2}

The following lemma is immediate.

\begin{lemma}\label{formofu}
Assume Assumption \ref{mainass} (2). Let $R \in \textup{LCN}(E)$. Then (up to strict equivalence) any crystalline uppertriangular deformation
$\rho$ of $\rho_0$ to $R$ must have the form
$$\rho = \bmat \rho_{1,R} & * \\ & \rho_{2,R}\emat,$$ where $\rho_{i,R}$ stands for the composite of $\tilde{\rho}_i$ with the $\Oo$-algebra
structure map $\Oo \rightarrow R$.
\end{lemma}

\begin{proof} This follows immediately from
Assumption \ref{mainass} (2). \end{proof}

Put $W=\Hom_{\Oo}(\tilde{\rho}_{2}, \tilde{\rho}_{1})\otimes E/\Oo$ and $W_n=\{x \in W: \varpi^n x =0\}$.

\begin{thm} \label{boundons}  Suppose there exists a positive integer $m$ such that $$\# H^1_{\Sigma}(F, W)
\leq \#\Oo/\varpi^m.$$ Then $\rho_0$ does not admit any upper-triangular crystalline deformations to $\GL_{n}(\Oo/\varpi^{m+1})$.
\end{thm}

\begin{proof}
    Let $\rho_{m+1}$ be such a block-uppertriangular deformation. By Lemma \ref{formofu} $\rho_{m+1}$ must have the form $$\rho_{m+1}=\bmat \tilde
    \rho_1 \mod{\varpi^{m+1}}& b \\ & \tilde \rho_2 \mod{\varpi^{m+1}}\emat.$$ Since $\rho_{m+1}$ is crystalline it gives rise to an element
    $\mathcal{E}$ in $H^1_{\Sigma}(F, W_{m+1})$. We claim that $\mathcal{E} \notin H^1_{\Sigma}(F, W_{m+1})[\varpi^m]$. Consider the following
    diagram:

$$\xymatrix{ &  (W_{m+1}/W_1)^{G_{\Sigma}} \ar[d]\\ H^1(F_{\Sigma},W_{m+1}) \ar[r] \ar[rd]_{\varpi^m} & H^1(F_{\Sigma},W_1) \ar[d]\\ &  H^1(F_{\Sigma},W_{m+1})}$$

    The vertical sequence is induced from the exact sequence $0 \to W_1 \to W_{m+1} \to W_{m+1}/W_1 \to 0$, the horizontal from $0 \to W_m \to
    W_{m+1} \overset{\varpi^m}{\to} W_1 \to 0$. 

  Note that $W_n \cong T/\varpi^n$ by $x \mapsto \varpi^n x$. This isomorphism is $G_{\Sigma}$-equivariant since the action is $\Oo$-linear. This implies that $W_2/W_1 \cong T/\varpi T \cong W_1$ as $G_{\Sigma}$-modules.
By our assumption that $\rho_1$ and $\rho_2$ are irreducible and non-isomorphic we know that $\Hom(\rho_2,\rho_1)^{G_{\Sigma}}=0$, so we get
  $$W_1^{G_{\Sigma}}=(W_2/W_1)^{G_{\Sigma}}=0.$$

    Note that $(W_{m+1}/W_1)^{G_{\Sigma}}=0$ follows from $(W_2/W_1)^{G_{\Sigma}}=0$ since $W_{m+1}$ surjects onto
    $W_2$ under multiplication by $\varpi^{m-1}$. Therefore, if $\varpi^m \mathcal{E}=0$ then $\mE$ would have to lie in the kernel of the
    horizontal
    map. This map corresponds, however, under the isomorphism of $W_k \cong T/\varpi^k T$, to the morphism $$H^1(F_{\Sigma},T/\varpi^{m+1}T)
    \overset{\mod{\varpi}}{\to} H^1(F_{\Sigma},T/\varpi T).$$  Hence the image of $\mathcal{E}$ under the horizontal map corresponds to the non-split extension
    given by $\rho_0$. This proves the claim.

    By the structure theorem of finitely generated modules over the PID $\Oo$, the module $H^1_{\Sigma}(F, W_{m+1})$ must be isomorphic to a direct
    sum of modules of the form $\Oo/\varpi^r$. Since $\mE \notin H^1_{\Sigma}(F, W_{m+1})[\varpi^m]$, the module $H^1_{\Sigma}(F, W_{m+1})$ must
    have
    a submodule isomorphic to $\Oo/\varpi^{m+1}$.
    We claim that $W_1^{G_{\Sigma}}=0$ also implies $H^0(F_{\Sigma},W)=0$. For this consider $a \in W^{G_{\Sigma}}$. If $a\neq 0$, then there exists $n$ such that
$\varpi^n a=0$ but $\varpi^{n-1}a \neq 0$. Since the $G_{\Sigma}$-action is $\Oo$-linear, $a \varpi^{n-1}$ lies in $W_1^G=0$, so $a=0$, which proves the claim. By the claim and Proposition \ref{functoriality}, $H^1_{\Sigma}(F,
    W_{m+1})=H^1_{\Sigma}(F, W)[\varpi^{m+1}]$. By our assumption on the bound on $\# H^1_{\Sigma}(F, W)$ this contradicts the existence of
    $\rho_{m+1}$.
\end{proof}

\begin{rem} The existence of an $m$ as in Theorem \ref{boundons} follows essentially from (the $\varpi$-part of) the Bloch-Kato conjecture for the
module $\Hom_{\Oo}(\tilde{\rho}_{2}, \tilde{\rho}_{1})$ and its value should equal the $\varpi$-adic valuation of a special $L$-value associated with
this module divided by an appropriate period. See also section \ref{Main assumptions} to see how one can deal with primes $v \in \Sigma \setminus \Sigma_p$. \end{rem}

\subsection{Cyclicity of $R_{\Sigma}/I_{\rm re}$} \label{Cyclicity}

\begin{thm} \label{repr23} Let $R$ be a local complete Noetherian $\Oo$-algebra with residue field $\bfF$. If $T: R[G_{\Sigma}] \rightarrow R$ is a
pseudocharacter such that $\ov{T}$ is the trace of a $d$-dimensional absolutely irreducible representation, then there exists a unique (up to
isomorphism) representation $\rho_T: G_{\Sigma} \rightarrow \GL_d(R)$ such that $\tr \rho_T = T$. \end{thm}

\begin{proof} This is Theorem 2.18 in \cite{Hida00}. \end{proof}

\begin{thm} \label{urbres} Let $(R, \fm_R, \bfF)$ be a local Artinian (or complete Hausdorff)
ring. Let $\sigma_1$, $\sigma_2$, and $\sigma$ be three
 representations of a topological group $G$ with coefficients in $R$ (with $\sigma$ having
image in $\GL_n(R)$).
Assume the following are true:
\begin{itemize}
\item $\sigma$ and $\sigma_1 \oplus \sigma_2$ have the same characteristic polynomials; \item The mod $\fm_R$-reductions $\ov{\sigma}_1$ and
    $\ov{\sigma}_2$ of $\sigma_1$ and $\sigma_2$ respectively are absolutely irreducible and non-isomorphic; \item The mod $\fm_R$-reduction
    $\ov{\sigma}$ of $\sigma$ is indecomposable and the subrepresentation of $\ov{\sigma}$ is isomorphic to $\ov{\sigma}_1$.
\end{itemize} Then there exists $g \in \GL_n(R)$ such that $$\sigma(h) = g
\bmat \sigma_1(h) & * \\ & \sigma_2(h)\emat g^{-1}$$ for all $h \in G$.
\end{thm}

\begin{proof} This is Theorem 1 in \cite{Urban99}. \end{proof}

\begin{cor} \label{equi2} Let
$I\subset R'_{\Sigma}$ be an ideal such that $R'_{\Sigma}/I \in \textup{LCN}(E)$ and is an Artin ring. Then
$I$ contains the ideal of reducibility of $R'_{\Sigma}$  if and only if $\rho'_{\Sigma}$ mod $I$ is an upper-triangular deformation of $\rho_0$ to
$\GL_{n}(R'_{\Sigma}/I)$.
\end{cor}

\begin{proof} If $\rho'_{\Sigma}$ mod $I$ is isomorphic to
an upper-triangular deformation of $\rho_0$ to $\GL_{n}(R'_{\Sigma}/I)$, then clearly $\tr \rho'_{\Sigma}$ mod $I$ is the sum of two traces
reducing to $\tr \rho_1 + \tr \rho_2$, so $I$ contains the ideal of reducibility. We will now prove the converse. Suppose $I$ contains the ideal of
reducibility. Then by definition $\tr \rho'_{\Sigma} = T_1 + T_2$ mod $I$ for two pseudocharacters $T_1, T_2$ such that $\ov{T}_i = \tr \rho_i$.
Since $\rho_i$ are absolutely  irreducible it follows from Theorem \ref{repr23} that there exist $\rho_{T_i}:R'_{\Sigma}/I[G_{\Sigma}] \rightarrow
R'_{\Sigma}/I$ such that $T_i = \tr \rho_{T_i}$ mod $I$. By \cite{BellaicheChenevierbook}, section 1.2.3 and the fact that $p \nmid n!$ one
has
$$\tr \rho'_{\Sigma} \pmod{I} = \tr \rho_{T_1} +\tr \rho_{T_2} = \tr (\rho_{T_1} \oplus \rho_{T_2}) \implies \chi_{\rho'_{\Sigma}\hspace{2pt} \textup{mod} \hspace{2pt}I} = \chi_{\rho_{T_1}\oplus
\rho_{T_2}},$$ where $\chi$ stands for the characteristic polynomial. By the Brauer-Nesbitt Theorem (or Theorem \ref{repr23} for $R=\bfF$) we
conclude that $\ov{\rho}_{T_i} \cong \rho_i$, so we can apply Theorem \ref{urbres} to get that $\rho'_{\Sigma}$ mod $I$ is isomorphic to a
block-upper-triangular representation, say $\sigma$. Using the fact that the map $(R'_{\Sigma})^{\times} \rightarrow (R'_{\Sigma}/I)^{\times}$ is
surjective we see that we can further conjugate $\sigma$ (over $R'_{\Sigma}/I$) to a block-upper-triangular \emph{deformation} of $\rho_0$.
\end{proof}

\begin{lemma} \label{x1x2} If $R$ is a local complete Noetherian
$\Oo$-algebra then it is a quotient of $\Oo[[X_1, X_2, \dots, X_s]]$.
\end{lemma}
\begin{proof} This is Theorem 7.16a,b of \cite{Eisenbud}. \end{proof}

\begin{prop} \label{comm1} Assume Assumption \ref{mainass} (1),
(2).  Then the structure map $\Oo \rightarrow R'_{\Sigma, \Oo}/I'_{\ture}$ is surjective.
 \end{prop}

Before we prove the proposition we will show that it implies the corresponding statement for $I_{\rm re}$ and $I^{\rm red}_{\rm re}$.
\begin{lemma} \label{redidfunct}
  Let  $$\xymatrix{A \ar@{>>}[rr]^{\varphi}\ar@{>>}[dr] && B \ar@{>>}[dl]\\ & \bfF &},$$  be a commutative diagram
of commutative $A$-algebras. Define $T_B$ via the commutative diagram  $$\xymatrix{A[G]\ar[r]^{T}\ar[d]_{\varphi} & A \ar[d]^{\varphi}\\ B[G]
\ar[r]^{T_B}& B}.$$ Then $\varphi$ induces a surjection $$A/I_T \twoheadrightarrow B/I_{T_B}.$$
\end{lemma}

\begin{proof}  It is enough to show that
$\varphi(I_{T})\subset I_{T_B}$. Indeed, assuming this, $\varphi$ induces a well-defined map $A/I_{T} \rightarrow B/I_{T_B}$, which must be a
surjection since $\varphi$ is. Since $A/\varphi^{-1}(I_{T_B}) \cong B/I_{T_B}$, we see that $T$ modulo $\varphi^{-1}(I_{T_B})$ is a sum of
pseudocharacters, hence $\varphi^{-1}(I_{T_B})\supset I_{T}$. Since $\varphi$ is a surjection it follows that $I_{T_B}\supset \varphi(I_T)$.
\end{proof}

\begin{cor} \label{structmap} Assume Assumption \ref{mainass} (1),
(2).  Then the structure maps $\Oo \rightarrow R_{\Sigma}/I_{\ture}$ and $\Oo \rightarrow R_{\Sigma}^{\rm red}/I^{\rm red}_{\rm re}$ are surjective.
\end{cor}

\begin{proof} This follows immediately from Proposition \ref{comm1}
 and Lemma \ref{redidfunct} where $A=R'_{\Sigma}$, $B=R_{\Sigma}$ or
$B=R_{\Sigma}^{\rm red}$. \end{proof}

\begin{proof} [Proof of Proposition \ref{comm1}] Write $S$ for $R'_{\Sigma}/I'_{\ture}$. Then $S$ is a
local complete ring. Moreover, by Lemma \ref{x1x2} we have that $S$ is a quotient of $\Oo[[X_1, \dots, X_s]]$, and hence $R'_{\Sigma}/\varpi
R'_{\Sigma}$ (and thus $S/\varpi  S$) is a quotient of $\bfF[[X_1, \dots, X_s]]$. We first claim that in fact $S/\varpi  S = \bfF$. Indeed, assume
otherwise, i.e., that $S/\varpi S=\bfF[[X_1, \dots, X_s]]/J$ and $s>0$, then $S/\varpi S$ admits a surjection, say $\phi$ onto
$\bfF[X]/X^2$,  i.e., there are at least two distinct
elements of $\Hom_{\Oo-\textup{alg}}(R'_{\Sigma}, \bfF[X]/X^2)$ - the map $R'_{\Sigma}\twoheadrightarrow \bfF \hookrightarrow \bfF[X]/X^2$ and the
surjection $R'_{\Sigma} \twoheadrightarrow \bfF[X]/X^2$ arising from $\phi$. By the definition of $R'_{\Sigma}$ there is a one-to-one correspondence
between the deformations to $\bfF[X]/X^2$ and elements of $\Hom_{\Oo-\textup{alg}}(R'_{\Sigma},\bfF[X]/X^2)$. The trivial element corresponds to the
trivial deformation to $\bfF[X]/X^2$, i.e., with image contained in $\GL_2(\bfF)$, which is clearly upper-triangular. However, the deformation
corresponding to the surjection must also be upper-triangular by Corollary \ref{equi2} since $\ker (R'_{\Sigma} \twoheadrightarrow S/\varpi
S\twoheadrightarrow \bfF[X]/X^2)$ contains $I'_{\ture}$ and $\bfF[X]/X^2$ is Artinian. But we know by Proposition \ref{nored} that $\rho_0$ does not
admit any non-trivial crystalline upper-triangular deformations to $\bfF[X]/X^2$. Hence we arrive at a contradiction. So, it must be the case
that $S/\varpi S=\bfF$.

Thus by the complete version of Nakayama's Lemma (\cite{Eisenbud}, Exercise 7.2) we know that $S$ is generated (as a $\Oo$-module) by one element.
\end{proof}

\begin{prop} \label{traces} The ring $R'_{\Sigma}$ is topologically generated as
an $\Oo$-algebra by the set $$S:=\{\tr\rho'_{\Sigma}(\Frob_v) \mid v \not\in \Sigma\}.$$ \end{prop}

\begin{proof}
Let $R_{\Sigma}^{\rm tr}$ be the closed (and hence complete by \cite{Matsumura} Theorem 8.1) $\Oo$-subalgebra of $R'_{\Sigma}$ generated by the set
$S$. Let $I_0^{\rm tr}$ be the smallest closed ideal of $R_{\Sigma}^{\rm tr}$ containing the set $$T:= \{ \tr \rho'_{\Sigma}(\Frob_v) - \tr
\tilde{\rho}_1(\Frob_v) - \tr \tilde{\rho}_2(\Frob_v) \mid v \not\in \Sigma \}. $$ Note that $\tr \rho'_{\Sigma}(\Frob_v) - \tr
\tilde{\rho}_1(\Frob_v) - \tr \tilde{\rho}_2(\Frob_v) \equiv 0$ (mod $\varpi$) for $v \not\in \Sigma$, so $I_0^{\rm tr}\neq R_{\Sigma}^{\rm tr}$.
Also note that $I_0:= I_0^{\rm tr}R'_{\Sigma}$ is the smallest closed ideal of $R'_{\Sigma}$ containing $T$. We will now show that $I_0 = I'_{\rm
re}$. Indeed, by the Chebotarev density theorem we get $\tr\rho'_{\Sigma} = \tr\tilde{\rho}_1 + \tr\tilde{\rho}_2$ (mod $I_0$), hence $I_0 \supset
I'_{\rm re}$. On the other hand since $R'_{\Sigma}/I'_{\rm re}$ is complete Hausdorff, we can apply Corollary \ref{equi2} to the ideal $I'_{\rm re}$
to conclude that $\rho'_{\Sigma}$ (mod $I'_{\rm re}$) is an upper-triangular deformation of $\rho_0$ and thus by Lemma \ref{formofu} we must have
$\tr \rho'_{\Sigma} = \tr \tilde{\rho}_1 + \tr \tilde{\rho}_2$ (mod $I'_{\rm re}$). It follows that $I_0 \subset I'_{\rm re}$.

Note that since $\tr \tilde{\rho}_i$ is $\Oo$-valued, the $\Oo$-algebra structure map $\Oo \rightarrow R_{\Sigma}^{\rm tr}/I_0^{\rm tr}$ is
surjective, hence $R_{\Sigma}^{\rm tr}/(I_0^{\rm tr}+ \varpi R_{\Sigma}^{\rm tr}) = \bfF$.
 Thus in particular $\fm^{\rm tr}:= I_0^{\rm tr}+ \varpi R_{\Sigma}^{\rm
tr}$ is the maximal ideal of $R_{\Sigma}^{\rm tr}$. Moreover, the containment \be \label{incl23}R_{\Sigma}^{\rm tr} \hookrightarrow R'_{\Sigma}\ee
gives rise to an $\Oo$-algebra map \be \label{incl24} R_{\Sigma}^{\rm tr}/I_0^{\rm tr} \rightarrow R'_{\Sigma}/I_0,\ee which must be surjective since
the object on the right equals $R'_{\Sigma}/I'_{\rm re}$ by the above argument and $R'_{\Sigma}/I'_{\rm re}$ is generated by 1 as an $\Oo$-algebra by
Proposition \ref{comm1}. This map descends to
$$\bfF=R_{\Sigma}^{\rm tr}/\fm^{\rm tr} \rightarrow R'_{\Sigma}/\fm^{\rm
tr}R'_{\Sigma} = R'_{\Sigma}/(I_0 + \varpi R_{\Sigma}^{\rm tr})=R'_{\Sigma}/(I'_{\rm re}+\varpi R_{\Sigma}^{\rm tr})=\bfF,$$ which is an isomorphism
since (\ref{incl24}) was surjective. Note that the maps (\ref{incl23}) and (\ref{incl24}) are in fact $R_{\Sigma}^{\rm tr}$-algebra maps and since
$R_{\Sigma}^{\rm tr}$ is complete (which means complete with respect to $\fm^{\rm tr}$) we can apply the complete version of Nakayama's lemma to
conclude that $R_{\Sigma}^{\rm tr}= R'_{\Sigma}$.
\end{proof}

\begin{prop} \label{Ovarpis} Assume Assumption \ref{mainass} and $\#
H^1_{\Sigma}(F, \Hom_{\Oo}(\tilde{\rho}_{2}, \tilde{\rho}_{1})) \leq \#\Oo/\varpi^m$. Then $ R'_{\Sigma}/I'_{\ture}=\Oo/\varpi^s$ for some $0< s \leq
m$. The same conclusion is true for $ R_{\Sigma}/I_{\ture}$ and for $ R_{\Sigma}^{\rm red}/I_{\ture}^{\rm red}$. \end{prop}
\begin{proof}
By Proposition \ref{comm1} we have that $R'_{\Sigma}/I'_{\ture}=\Oo/\varpi^s$ for some $s \in \bfZ_+ \cup \{\infty\}$. But we must have $0<r\leq m$,
since by Corollary \ref{equi2} if $r>m$ or $r=\infty$, then there would be an upper-triangular crystalline deformation of $\rho_0$ to
$\Oo/\varpi^{m+1}$, which is impossible by Theorem \ref{boundons}. The last assertion of the Proposition follows from Lemma \ref{redidfunct}.
\end{proof}

\subsection{Some consequences of the principality of $I_{\rm re}$} \label{Consequences}

Below we list some consequences of principality of $I_{\rm re}$ in our context.

\begin{lemma} \label{x1} If $R$ is a local complete Noetherian $\Oo$-algebra and there exists $r \in R$ such that the structure map $\Oo \rightarrow
R/rR$ is surjective, then $R$ is a quotient of $\Oo[[X]]$. \end{lemma}

\begin{proof} Since $R/(r, \varpi) = \bfF$, the ideal $(r,\varpi) \subset R$ is maximal. Hence by Theorem 7.16 in \cite{Eisenbud} there exists an
$\Oo$-algebra map $\Phi: \Oo[[X,Y]] \twoheadrightarrow R$ sending $X$ to $r$ and $Y$ to $\varpi$. But this map factors through $\Psi: \Oo[[X,Y]]
\twoheadrightarrow \Oo[[X]]$ sending $X$ to $X$ and $Y$ to $\varpi$ (indeed, $\ker \Psi = (Y-\varpi)\Oo[[X,Y]] \subset \ker \Phi$). \end{proof}

\begin{prop} \label{Gorenstein} If $R^{\rm red}_{\Sigma}/I^{\rm red}_{\rm re} \neq \Oo$, then $R^{\rm red}_{\Sigma}$ is Gorenstein. \end{prop}
\begin{proof} First note that by Proposition \ref{comm1} our assumption
implies that $R_{\Sigma}^{\rm red}/I^{\rm red}_{\rm re}$ is finite.
Thus by Corollary \ref{nonzerodivisor} the ideal $I^{\rm red}_{\rm re}$ is generated by a non-zero divisor. Hence in
particular the maximal ideal of $R^{\rm red}_{\Sigma}$ contains a non-zerodivisor. Thus we can apply \cite{Bass63}, Proposition 6.4 to conclude that
$R^{\rm red}_{\Sigma}$ is Gorenstein. \end{proof}

\begin{prop} If $R^{\rm red}_{\Sigma}/I^{\rm red}_{\rm re} \neq \Oo$, then
$R^{\rm red}_{\Sigma}$ is a complete intersection. \end{prop}

\begin{proof} By Lemma \ref{x1} we know that $R^{\rm red}_{\Sigma}=\Oo[[X]]/J$. Note that $\textup{codim} (J) = \dim R^{\rm red}_{\Sigma}$ which
because $I^{\rm red}_{\rm re}$ is principal equals (cf. e.g. \cite{AtiyahMacdonald}, Corollary 11.18) $\dim R^{\rm red}_{\Sigma}/I^{\rm red}_{\rm re}
+ 1 = 1$ since $R_{\Sigma}/I_{\rm re}$ is finite. It follows from \cite{Eisenbud}, Corollary 21.20 that $R^{\rm red}_{\Sigma}$ is a complete
intersection.
\end{proof}

\section{A commutative algebra criterion} \label{criterion}

Let $R$ and $S$ denote complete local Noetherian $\Oo$-algebras with residue field $\bfF$. Suppose that $S$ is a finitely generated free
$\Oo$-module.

\begin{thm} \label{cri1} Suppose there exists a surjective $\Oo$-algebra
map $\phi: R\twoheadrightarrow S$ inducing identity on the residue fields  and $\pi \in R$ such that the following diagram \be
\label{cridiag}\xymatrix{R \ar[r]^{\phi} \ar[d] & S \ar[d]
\\ R/\pi R \ar[r] & S/\phi(\pi) S}\ee commutes.
Write $\phi_n$ for the map $\phi_n: R/\pi^n R \twoheadrightarrow S/\phi(\pi)^n S$.
 Assume $\# \phi(\pi)S/\phi(\pi)^2 S < \infty$.
Suppopse $\phi_1: R/\pi R \twoheadrightarrow S/\phi(\pi)S$ is an isomorphism.
\begin{itemize}
 \item If $ R/\pi R \cong \Oo/\varpi^r$ for some positive integer $r$, then $\phi$ is an isomorphism.
\item If $ R/\pi R \cong \Oo$ and the induced map $\pi R/\pi^2 R \twoheadrightarrow \phi(\pi)S/\phi(\pi)^2S$ is an isomorphism, then $\phi$ is
    an
    isomorphism.
\end{itemize}
 \end{thm}

In the case when $R/\pi R =\Oo$, Theorem \ref{cri1} gives an alternative to the criterion of Wiles and Lenstra to prove $R=T$. Let us briefly recall
this criterion. Suppose we have the following commutative diagram of surjective $\Oo$-algebra maps \be \label{diagWL} \xymatrix{R \ar[r]^{\phi}
\ar[d]^{\pi_R} & S \ar[dl]^{\pi_S}\\ \Oo}\ee and for $A=R$ or $S$ set $\Phi_A:=(\ker \pi_A)/(\ker \pi_A)^2$ and $\eta_A= \pi_A(\Ann_A \ker \pi_A)$.
\begin{thm} [Wiles and Lenstra] $\# \Phi_R \leq \# \Oo/\eta_S$ if and only if $\phi$ is an isomorphism of complete intersections. \end{thm}

\begin{prop} \label{lenstraredone} Suppose diagram (\ref{diagWL}) commutes. Suppose that $\ker \pi_R$ is a principal ideal of $R$ generated by some
$\pi \in R$ and suppose that $\# \Phi_R \leq \# \Oo/\eta_S$, then $\phi$ is an isomorphism. \end{prop}

\begin{proof} By our assumption we have \be \label{eq648}\# \pi R/ \pi^2 R
= \# \Phi_R \leq \#\Oo/\eta_S.\ee On the other hand the right-hand-side of (\ref{eq648}) is bounded from above by $\# \Phi_S$ (see e.g. formula
(5.2.3) in \cite{DDT}). However, note that since $\phi$ is surjective it follows that $\phi(\ker \pi_R) = \ker \pi_S$, hence
$\Phi_S=\phi(\pi)S/\phi(\pi)^2S$. Hence we can apply Theorem \ref{cri1} to conclude that $\phi$ is an ismorphism. \end{proof}

\begin{proof} [Proof of Theorem \ref{cri1}]
Consider the following commutative diagram with exact rows. \be \label{diag3} \xymatrix{0\ar[r]&\pi R/\pi^n R\ar[r]\ar[d]^{\alpha}& R/\pi^n R
\ar[r]\ar[d]^{\beta} & R/\pi R \ar[r]\ar[d]^{\phi_1} & 0 \\ 0 \ar[r] &\phi(\pi) S / \phi(\pi)^{n}S \ar[r] & S/\phi(\pi)^n S \ar[r] & S/\phi(\pi) S
\ar[r]& 0}\ee
 We will show that $R/\pi^n R
\cong S/\phi(\pi)^n S$ for all $n$. By (\ref{diag3}) and snake lemma it is enough to show that $\alpha$ is an isomorphism for all $n>1$.

Set $x=\phi(\pi)$. Note that $\alpha$ is clearly surjective, because $\phi$ is. On the other hand, the multiplication by $\pi$ (resp. by $x$) induces
surjective maps: $\pi^{k-1}R/\pi^k R \twoheadrightarrow \pi^k R/\pi^{k+1}R$ (resp. $x^{k-1}S/x^kS \twoheadrightarrow x^k S/ x^{k+1} S$). So, arguing
as in the proof of Proposition 6.9 in \cite{BergerKlosin11} we have $\#(\pi R/\pi^k R) \leq \#(\pi R/\pi^2 R)^{k-1}$ and $\#(x S/x^k R) =
\#(x S/x^2 S)^{k-1}$ because by Lemma 6.7 of \cite{BergerKlosin11} we get that the multiplication by $x$ is injective on $xS$ (apply this
lemma for $xS$ instead of $S$ - note that $xS$ being a submodule of a finitely generated torsion free $\Oo$-module is also finitely generated and
torsion-free). If $R/\pi R \cong \Oo$, then $\pi R/\pi^2 R \cong xS/x^2S$ by assumption. If $R/\pi R =\Oo/\varpi^r \Oo$ for some positive integer $r$
we deduce that $\pi R/\pi^2 R \cong xS/x^2S$ as in the proof of Proposition 6.9 in [loc.cit.] and (arguing inductively) in both cases we finally
obtain $\pi R/\pi^k R \cong xS/x^k S$, which is what we wanted. So,
$$\invlim_n R/\pi^n R \cong \invlim_n S/x^n S.$$

Now, consider the following commutative diagram with exact rows
$$\xymatrix{0 \ar[r] & R \ar[r]^{\iota} \ar[d]^{\phi} & \invlim_n R/\pi^n R \ar[r]\ar[d]^{\phi}_{\sim} & \coker \iota
\ar[d]\ar[r] & 0 \\ 0 \ar[r] & S \ar[r]^{\iota'} & \invlim_n S/\phi(\pi)^n S \ar[r] & \coker \iota' \ar[r] & 0}$$ where the maps $\iota$ and $\iota'$
are injective because $R$ (resp. $S$) are separated (with respect to the maximal ideals hence with respect to any non-unit ideals). The first
vertical map is surjective and the second is an isomorphism, hence by snake lemma the first vertical map is an isomorphism as well. \end{proof}

\section{$R=T$ theorems} \label{RTtheorems}

 Fix a (semi-simple) $p$-adic Galois representation $\rho_{\pi_0}: G_F \rightarrow \GL_n(E)$ which factors through $G_{\Sigma}$ and satisfies:
\be \label{semis} \ov{\rho}_{\pi_0}^{\rm ss} \cong \rho_1 \oplus \rho_2.\ee

\begin{prop} \label{genRibet} If $\rho_{\pi_0}$ is irreducible then there exists a lattice $\mL$ inside $E^{n}$ so that with respect to that lattice
the mod $\varpi$ reduction $\ov{\rho}_{\pi_0}$ of $\rho_{\pi_0}$ has the form $$\ov{\rho}_{\pi_0}=\bmat \rho_1 & * \\ 0 & \rho_2\emat$$ and is
non-semi-simple. \end{prop}

\begin{proof} This is a special case of \cite{Urban01}, Theorem 1.1, where the ring $\mB$ in [loc.cit.] is a discrete valuation ring $=\Oo$.
\end{proof}

Set $$\rho_0:= \ov{\rho}_{\pi_0}.$$

Let $\Pi$ be the set of Galois representations $\rho_{\pi}: G_{\Sigma} \rightarrow \GL_n(E)$ (with $\rho_{\pi}$ semi-simple but not necessarily irreducible) for which there exists a
crystalline deformation $\rho'_{\pi}: G_{\Sigma} \rightarrow \GL_n(\Oo)$ of $\rho_0$ such that one has $$(\rho'_{\pi})^{\rm ss} \cong_{/E}
\rho_{\pi}.$$

\begin{rem} \label{conntohecke} Our choice of notation is motivated by potential applications of these results. In applications $\rho_{\pi_0}$ will
be the Galois representation attached to some automorphic representation $\pi_0$ and $\Pi$ will be (in one-to-one correspondence with) the subset of
($L$-packets of) automorphic representations $\pi$ whose associated Galois representation $\rho_{\pi}$ satisfies the above condition. \end{rem}
\begin{prop} \label{genRibet2} Assume Assumption \ref{mainass}(1). If $\rho: G_{\Sigma} \rightarrow \GL_n(E)$ is irreducible and crystalline and $\ov{\rho}^{\rm ss} = \ov{\rho}_0^{\rm ss}$, then $\rho \in \Pi$. \end{prop}

\begin{proof} By Proposition \ref{genRibet} $\rho$ is $E$-isomorphic to a representation $\rho': G_{\Sigma} \rightarrow
\GL_n(\Oo)$ with $\ov{\rho}'$ upper-triangular and non-semi-simple. Since $\rho'$ is crystalline its reduction gives rise to a non-zero element inside $H^1_{\Sigma}(F, \Hom_{\bfF}(\rho_2, \rho_1))$ and by
Assumption \ref{mainass}(1) this group is one-dimensional. \end{proof}

\begin{rem} In contrast to Proposition \ref{genRibet2} if $\rho$ is reducible (and by assumption semi-simple) it is not always the case that $\rho \in \Pi$. For example Skinner and Wiles in \cite{SkinnerWiles99} studied a minimal (ordinary) deformation problem for residually reducible 2-dimensional Galois
representations. In [loc.cit] they assert the existence of an upper-triangular $\Sigma$-minimal deformation $\rho'$ of $\rho_0$ to $\GL_2(\Oo)$ based
on arguments from Kummer theory. The semi-simplification of this deformation is the Galois representation $\rho_{E_{2, \varphi}}$ associated to a
certain Eisenstein series $E_{2, \varphi}$ (see page 10523 in [loc.cit.] for a definition of $E_{2, \varphi}$), hence we take $\rho'_{E_{2,
\varphi}}=\rho'$. The difficulty is in showing the existence of a representation  whose semi-simplification agrees with $\rho_{E_{2, \varphi}}$, but
which reduces to $\rho_0$ hence is non-semi-simple (this is where the Kummer theory is used). In contrast to the case considered in
\cite{SkinnerWiles99}, the authors showed that in the case of 2-dimensional Galois representations over an imaginary quadratic field $F$ there is no
upper-triangular $\Sigma$-minimal deformation of $\rho_0$ to $\GL_2(\Oo)$ (\cite{BergerKlosin09}, Corollary 5.22). So, in particular if one considers
an Eisenstein series (say $\mE$) over $F$ then there is no representation $\rho'_{\mE}$ whose semi-simplification is isomorphic to $\rho_{\mE}$ and
for which one has $\ov{\rho}'_{\mE} = \rho_0$ and at the same time $\rho'_{\mE}$ is minimal.  \end{rem}

Let $\Pi$ be as above. Then one obtains an $\Oo$-algebra map $$R_{\Sigma} \rightarrow \prod_{\rho_{\pi} \in \Pi} \Oo.$$ We (suggestively) write
$\bfT_{\Sigma}$ for the image of this map and denote the resulting surjective $\Oo$-algebra map $R_{\Sigma} \twoheadrightarrow \bfT_{\Sigma}$ by
$\phi$.

\begin{thm} \label{cons1} Suppose the set $\Pi$ is finite. Assume Assumption \ref{mainass} (1) and (2).
Suppose there exists an anti-automorphism $\tau$ of $R_{\Sigma}[G_{\Sigma}]$ such that $\tr \rho_{\Sigma} \circ \tau = \tr \rho_{\Sigma}$ and $\tr
\rho_i \circ \tau = \tr \rho_i$ for $i=1,2$. In addition suppose that there exists a positive integer $m$ such that the following two ``numerical''
conditions are satisfied:
\begin{enumerate}
\item $\# H^1_{\Sigma}(F, \Hom_{\Oo}(\tilde{\rho}_2, \tilde{\rho}_1) \leq \#\Oo/\varpi^m$, \item $\#\bfT_{\Sigma}/\phi(I_{\rm re}) \geq
    \#\Oo/\varpi^m$.\end{enumerate} Then the map $\phi: R_{\Sigma} \twoheadrightarrow \bfT_{\Sigma}$ is an isomorphism.
\end{thm}

\begin{proof} This is just a summary of our arguments so far. The
existence of $\tau$ guarantees principality of the ideal of reducibility $I_{\rm re}$. Condition (1) in Theorem \ref{cons1} implies (by Proposition
\ref{Ovarpis}) that $\#R_{\Sigma}/ I_{\rm re} = \Oo/\varpi^s$ for $0< s \leq m$. This combined with condition (2) guarantees that $\phi$ descends to
an isomorphism $\phi_1: R_{\Sigma}/ I_{\rm re} \rightarrow \bfT_{\Sigma}/\phi(I_{\rm re})$. Hence by Theorem \ref{cri1}(1) we conclude that $\phi$ is
an isomorphism. \end{proof}

\begin{thm} \label{cons2} Suppose the set $\Pi$ is finite. Assume Assumption \ref{mainass}. Suppose there exists an anti-automorphism $\tau$ of
$R_{\Sigma}[G_{\Sigma}]$ such that $\tr \rho_{\Sigma} \circ \tau = \tr \rho_{\Sigma}$ and $\tr \rho_i \circ \tau = \tr \rho_i$ for $i=1,2$. In
addition suppose that \begin{enumerate} \item $\bfT_{\Sigma}/\phi(I_{\rm re}) = \Oo$, \item $\# I_{\rm re}/(I_{\rm re})^2 \leq \#(\phi(I_{\rm re})
\bfT_{\Sigma})/(\phi(I_{\rm
    re})\bfT_{\Sigma})^2$. \end{enumerate} Then the map $\phi: R_{\Sigma} \twoheadrightarrow \bfT_{\Sigma}$ is an isomorphism. \end{thm}

\begin{proof} This is proved analogously to Theorem \ref{cons2} but uses Theorem \ref{cri1}(2). Note that Corollary \ref{structmap} combined with
condition (1) of Theorem \ref{cons2} yields $ R_{\Sigma}/I_{\rm re} = \bfT_{\Sigma}/\phi(I_{\rm re}) = \Oo$, hence the map $\phi_1$ in Theorem
\ref{cri1} is an isomorphism. \end{proof}

\begin{rem} In applying Theorems \ref{cons1} and \ref{cons2} in practice one identifies $\bfT_{\Sigma}$ with a local complete Hecke algebra. Then
condition (2) may be a consequence of a lower bound on the order of $\bfT_{\Sigma}/J$, where $J$ could be the relevant congruence ideal (e.g.,
Eisenstein ideal - see section \ref{Examples1} or Yoshida ideal - see section \ref{Examples2}). See for example \cite{BergerKlosin09} and
\cite{BergerKlosin11}, where such a condition (which is a consequence of a result proved in \cite{B09}) is applied in the context of Theorem
\ref{cons1}.
On the other hand in \cite{SkinnerWiles99} one shows that the
condition needed to apply the criterion of Wiles and Lenstra is satisfied and this implies (cf. Proposition \ref{lenstraredone} and its proof) that
condition (2) in Theorem \ref{cons2} is satisfied. On the other hand condition (1) of Theorem \ref{cons1} seems to require (the $\varpi$-part of) the
Bloch-Kato conjecture for $\Hom_{\Oo}(\tilde{\rho}_2, \tilde{\rho}_1)$ and is in most cases when $\rho_1$ and $\rho_2$ are not characters currently
out of reach. Hence in this case Theorem \ref{cons1} should be viewed as a statement asserting that under certain assumptions, (the $\varpi$-part of)
the Bloch-Kato conjecture for $\Hom_{\Oo}(\tilde{\rho}_2, \tilde{\rho}_1)$ (which in principle controls extensions of $\tilde{\rho_2}$ by
$\tilde{\rho}_1$ hence \emph{reducible} deformations of $\rho_0$) implies an $R=T$-theorem (which asserts modularity of both the reducible and the
irreducible deformations of $\rho_0$). \end{rem}

\begin{rem} If an anti-automorphism $\tau$ in Theorems \ref{cons1} and \ref{cons2} does not exist, but instead one has $$\dim_{\bfF}H^1(G_{\Sigma},
\Hom_{\bfF}(\rho_1, \rho_2)) = \dim_{\bfF}H^1(G_{\Sigma}, \Hom_{\bfF}(\rho_2, \rho_1))=1,$$ then the conclusions of Theorems  \ref{cons1} and
\ref{cons2} still hold by Remark \ref{noinvolution}. \end{rem}

\section{2-dimensional Galois representations of an imaginary quadratic field - the crystalline case} \label{Examples1}

In this and in the next section we will describe how the method outlined in the preceding sections can be applied in concrete situations. We begin
with the case when $F$ is an imaginary quadratic field, $\rho_1$ and $\rho_2$ are characters. This is similar to the problem studied in
\cite{BergerKlosin11}, but covers the case of crystalline deformations (as opposed to ordinary minimal deformations considered in [loc.cit.]).
Because of this similarity with \cite{BergerKlosin11}, we will discuss only the aspects in which this case differs from the ordinary case and
will refer the reader to \cite{BergerKlosin09} and \cite{BergerKlosin11} for most details and definitions. In the next section we will study
another case, this time when the representations $\rho_1$ and $\rho_2$ are 2-dimensional and will consider an $R=T$ problem for residually reducible
$4$-dimensional Galois representations of $G_{\bfQ}$.

\subsection{The setup} Let $F$ be an imaginary
quadratic extension of $\bfQ$ of discriminant $d_F \neq 3,4$ and $p>3$ a rational prime which is unramified in $F$. We fix once and for all a prime
$\fp$ of $F$ lying over $(p)$.  As before, we fix for every prime $\fq$ embeddings $\ov{F} \hookrightarrow \ov{F}_{\fq} \hookrightarrow \bfC$ and write $D_{\fq}$ and $I_{\fq}$ for the corresponding decomposition and inertia subgroups. We assume that $p \nmid \# \Cl_F$ and that any prime $q \mid
d_F$ satisfies $q \not\equiv \pm 1 \pmod{p}.$

Let $\Sigma$ be a finite set of finite primes of $F$ containing all the primes lying over $p$. Let $\chi_0: G_{\Sigma} \rightarrow \bfF^{\times}$ be a Galois character and
$$\rho_0= \bmat 1 & * \\ & \chi_0 \emat : G_{\Sigma} \rightarrow \GL_2(\bfF)$$ be a non-semi-simple Galois
representation.

\subsection{Assumption \ref{mainass}}
We will now describe sufficient conditions under which Assumption \ref{mainass} is satisfied.

 Let $S_p$ be the set of primes of $F(\chi_0)$ lying over
$p$.  Write $M_{\chi_0}$ for $\prod_{\fq \in S_p} (1+ \fP_v)$ and $T_{\chi_0}$ for its torsion submodule. The quotient $M_{\chi_0}/T_{\chi_0}$ is a
free $\bfZ_p$-module of finite rank. Let $\ov{\mE}_{\chi_0}$ be the closure in $M_{\chi_0}/T_{\chi_0}$ of the image of $\mE_{\chi_0}$, the group of
units of the ring of integers of $F(\chi_0)$ which are congruent to 1 modulo every prime in $S_p$.

\begin{definition} \label{adm} We say that $\chi_0: G_{\Sigma} \rightarrow \bfF^{\times}$ is \emph{$\Sigma$-admissible} if it satisfies all of the
following conditions:
\begin{enumerate}
\item $\chi_0$ is ramified at $\fp$; \item if $\fq \in \Sigma$, then either $\chi_0$ is ramified at $\fq$ or $\chi_0(\Frob_{\fq}) \neq (\#
    k_{\fq})^{\pm 1}$ (as elements of $\bfF$);
\item if $\fq \in \Sigma$, then $\# k_{\fq} \not\equiv 1$ (mod $p$);
\item $\chi_0$ is anticyclotomic, i.e., $\chi_0(c \sigma c) = \chi_0(\sigma)^{-1}$ for every
    $\sigma \in G_{\Sigma}$ and $c$ the generator of $\Gal(F/\bfQ)$; \item the $\bfZ_p$-submodule $\ov{\mE}_{\chi_0} \subset
    M_{\chi_0}/T_{\chi_0}$ is saturated with respect to the ideal $p\bfZ_p$, \item The $\chi_0^{-1}$-eigenspace of the $p$-part of
    $\Cl_{F(\chi_0)}$ is trivial. \end{enumerate}
\end{definition}
\begin{rem} \label{invadm}
Note that $\chi_0$ is $\Sigma$-admissible if and only if $\chi_0^{-1}$ is (cf. Remark 3.3 in \cite{BergerKlosin09}). \end{rem}

\subsubsection{Assumption \ref{mainass}(1)}

Set $G=\Gal(F(\chi_0)/F)$. Let $L$ denote the maximal abelian extension of $F(\chi_0)$ unramified outside the set $\Sigma$ and such that $p$ annihilates $\Gal(L/F(\chi_0))$.
Then $V:= \Gal(L/F(\chi_0))$ is an $\bfF_p$-vector space endowed with an $\bfF_p$-linear action of $G$, and one has $$V \otimes_{\bfF_p} \ov{\bfF}_p
\cong \bigoplus_{ \varphi \in \Hom (G, \ov{\bfF}_p^{\times})} V^{\varphi},$$ where for a $\bfZ_p[G]$-module $N$ and an $\ov{\bfF}_p$-valued character
$\varphi$ of $G$, we write \be \label{eigen1} N^{\varphi} = \{n \in N \otimes_{\bfZ_p} \ov{\bfF}_p \mid \sigma n = \varphi (\sigma) n \hf \textup{for
every} \hs \sigma \in G\}.\ee Note that $V_0 \otimes_{\bfF_p} \ov{\bfF}_p$ is a direct summand of $V^{\chi_0^{-1}}$.

\begin{prop}\label{essuni} One has $\dim_{\ov{\bfF}_p} V^{\chi_0^{-1}} = 1$.
\end{prop}

\begin{proof} If $p$ is a split prime this assertion has been proved in \cite{BergerKlosin09} (see Theorem 3.5). For an inert $p$ the proof is essentially the same, so let us just point out how to reconcile some of the issues that arise in the inert case (for notation we refer the reader to the proof of Thereom 3.5 in [loc.cit]). In particular as opposed to the split case, in the inert case one gets that for every $\psi \in G^{\vee}$,
$$\dim_{\ov{\bfF}_p} (M/T)^{\psi} = 2.$$
For this one can argue as follows: Since the ramification index
of $p$ in $F(\chi_0)$ is no greater than $p^2-1$, the $p$-adic logarithm gives a $D_v$-equivariant isomorphism $\fP_v^{p+2} \cong 1+\fP_v^{p+2}$ for
every $v \mid p$. This followed by the injection $1+\fP_v^{p+2}\hookrightarrow 1+\fP_v$ yields an isomorphism of $G$-modules $\bigoplus_{v\mid
p}\fP_v^{p+2}\otimes \ov{\bfQ}_p \cong (M/T)\otimes \ov{\bfQ}_p$. It is not difficult to see that $$\prod_{\fq \in S_p} \fP_v^{p+2} \otimes
\ov{\bfQ}_p \cong \bigoplus_{\phi \in \Gal(F(\chi_0)/\bfQ)^{\vee}} \ov{\bfQ}_p(\phi)\cong \bigoplus_{\phi \in G^{\vee}} \ov{\bfQ}_p(\phi)\oplus
\ov{\bfQ}_p(\phi),$$ where $\ov{\bfQ}_p(\phi)$ denotes the one-dimensional $\ov{\bfQ}_p$-vector space on which $G$ (or $\Gal(F(\chi_0)/\bfQ)$) acts
via $\phi$. The claim follows easily from this. However, since we now only have one prime of $F$ lying over $p$, this still gives us (as in the split case) that $$((M/T)\otimes \ov{\bfF}_p)/(\ov{\mE}
\otimes \ov{\bfF}_p) \cong \ov{\bfF}_p(\mathbf{1}) \oplus \ov{\bfF}_p(\mathbf{1}) \oplus \bigoplus_{\psi \in G^{\vee}\setminus \{\mathbf{1}\}}
\ov{\bfF}_p(\psi).$$ Since $\chi_0 \neq \mathbf{1}$  we are done.
\end{proof}

As in the proof of Corollary 3.7 in \cite{BergerKlosin09} Proposition \ref{essuni} implies that the space $H^1(G_{\Sigma},\ov{\bfF}_p(\chi_0^{-1}))$
is one-dimensional and hence we obtain the following corollary (note that $\rho_0$ itself is crystalline, so the extension it gives rise to lies
in the Selmer group).

\begin{cor} \label{satisfy1} The pair $(1, \chi_0)$ for a $\Sigma$-admissible character $\chi_0$ satisfies Assumption \ref{mainass}(1). \end{cor}

\subsubsection{Assumption \ref{mainass}(2)} Write $\rho_i$ for the character $1$ or $\chi_0$.
\begin{prop} \label{noredi} There does not exist any non-trivial
crystalline deformation of $\rho_i$ to $\GL_1(\bfF[x]/x^2)$.
 \end{prop}

\begin{proof} Let $\rho: G_\Sigma \rightarrow \GL_1(\bfF[x]/x^2)$ be a crystalline deformation of $\rho_i$. Then since $\rho_i^{-1}$ is also crystalline we can without loss of generality assume that $\rho$ has the form $\rho= 1+ x \alpha$  for $\alpha:
G_\Sigma \rightarrow \bfF^+$ a group homomorphism (here $\bfF^+$ denotes the additive group of $\bfF$).

Let $\fq$ be a prime of $F$ and consider the restriction of $\alpha$ to $I_{\fq}$. If $\fq \in \Sigma$, $\fq \nmid p$ then $\#k_v \not\equiv 1$ mod
$p$ by Definition \ref{adm}(3), and thus one must have (by local class field theory) that $\alpha (I_{\fq})=0$. Thus $\alpha$ can only be ramified at the primes lying over $p$. The proposition thus follows
easily from the following lemma and the assumption that $p \nmid \# \Cl_F$. \end{proof}

\begin{lemma} \label{crisunr} A $p$-power order crystalline character $\psi: G_{\Sigma}
\rightarrow (\bfF[x]/x^2)^{\times}$
 must be unramified at primes lying above $p$.
\end{lemma}

\begin{proof} Since a character as above can be thought of as a $2$-dimensional representation $\rho: G_{\Sigma} \rightarrow
\GL_2(\bfF)$ of the form $$\rho(\sigma) = \bmat 1 & 0 \\ \alpha & 1 \emat,$$ it is enough to show that for $\fq$ lying over $p$ a crystalline
extension of the trivial one-dimensional $\bfF$-representation of $G_{F_{\fq}}$ by itself must be unramified at primes lying over $p$. However, such
an extension is necessarily split by Remark 6.13, p.589 of \cite{FontaineLaffaille82}. \end{proof}

\begin{cor}\label{satisfy2} The pair $(1, \chi_0)$ for a $\Sigma$-admissible character $\chi_0$ satisfies Assumption \ref{mainass}(2). \end{cor}

\subsection{Bounding the Selmer group} \label{Bounding the Selmer group} From now on we will make a particular choice of $\chi_0$ and $\Sigma$. Let
$\phi_1$, $\phi_2$ be two Hecke characters of infinity types $z$ and $z^{-1}$ respectively, and set $\phi=\phi_1/\phi_2$. Let $\phi_{\fp}$ denote the
$\fp$-adic Galois character corresponding to $\phi$. Set $\Psi:= \phi_{\fp} \epsilon$ and $\chi_0=\ov{\Psi}$. Assume that $\Sigma$ contains all the
primes dividing $M_1M_2M_1^cM_2^c\textup{disc}_Fp$, where $M_i$ denotes the conductor of $\phi_i$.

 Let $L^{\rm int}(0,\phi)$ be the special $L$-value attached to $\phi$ as in \cite{BergerKlosin09}. Write $W$ for $\Hom_{\Oo}(\Psi, 1)\otimes
 E/\Oo$.

 \begin{conj} \label{BK1} $\# H^1_f(F, W) \leq \#
\Oo/\varpi^m$, where $m=\val_{\varpi}(L^{\tuint}(0, \phi))$.\end{conj}

\begin{rem} Conjecture \ref{BK1} can in many
cases be deduced from the Main conjecture proven by Rubin \cite{Rubin91}. If
$\phi^{-1}=\psi^2$ for $\psi$ a Hecke character associated to a CM elliptic
curve, then one can argue as follows. By Proposition 4.4.3 in \cite{Deepreprint} and using that $H^1_f(F, W) \cong H^1_f(F, W^c)$, we have
$\# H^1_f(F, W)= \#H^1_f(F,E/\Oo(\phi_{\fp}^{-1}))$.  Thus we can use Corollary 4.3.4 in
\cite{Deepreprint} which together with the functional equation satisfied by $L(0,\phi)$ implies the desired inequality. \end{rem}

\begin{cor} \label{Eulerfact}  Assume that $\chi_0$ is $\Sigma$-admissible and that Conjecture \ref{BK1} holds for $\phi$. Then $\# H^1_{\Sigma}(F,W) \leq \#
\Oo/\varpi^m$, where $m=\val_{\varpi}(L^{\tuint}(0, \phi))$.\end{cor}
\begin{proof} Let $v \in \Sigma \setminus \Sigma_p$. First note that since $\Psi$ is $\Oo^{\times}$-valued one must have $W^{I_v}=W$ or $W^{I_v}=0$. So, in particular $W^{I_v}$ is divisible. Hence by Remark \ref{tamagawaremark}, we get that $\Tam^0_v(T^*)=1$, so by Lemma \ref{Tamagawa} it is enough to show that $P_v(V^*,1)\in \Oo^{\times}$. Let $\Sigma_{\rm un}$ be the subset of $\Sigma\setminus \Sigma_p$ consisting of those primes $v$ for which $\chi_0$ is unramified. If $v \not\in \Sigma_{\rm un}$, then this Euler factor is 1. Otherwise one has $$P_v(V^*,1)^{-1} = 1-\Psi\epsilon(\Frob_v)  \equiv 1-\chi_0(\Frob_v)\cdot \# k_v \pmod{\varpi}.$$ Because $\chi_0$ is $\Sigma$-admissible (cf. Definition \ref{adm}(2)) we are done by Conjecture \ref{BK1}.  \end{proof}

From now on assume that $\chi_0$ is $\Sigma$-admissible and that Conjecture \ref{BK1} holds for $\phi$.
Let $R_{\Sigma}$ denote the crystalline universal deformation ring of $\rho_0$ and $I_{\rm re}$ its ideal of reducibility.

\begin{cor} \label{satisfy3} One has $R_{\Sigma}/I_{\rm re} = \Oo/\varpi^s$ for some $0<s\leq m$, with $m$ as above. \end{cor}

\begin{proof} This follows from Conjecture \ref{BK1} and Proposition \ref{Ovarpis}. \end{proof}

\begin{rem} Note that this proof of Corollary \ref{satisfy3} differs from (and is simpler than) the proof of Theorem 5.12 in \cite{BergerKlosin09} in that we do not need to relate the Selmer groups to Galois groups. This is so because the proof of Theorem \ref{boundons} interprets (which is perhaps more natural) upper-triangular deformations directly as cohomology classes in the Selmer group. \end{rem}

\subsection{Modularity of crystalline residually reducible 2-dimensional Galois representations over $F$}

\begin{rem} \label{symmetry2} By Remark \ref{invadm} and Corollary \ref{satisfy1} one also has $\dim_{\bfF}H^1(G_{\Sigma},\ov{\bfF}_p(\chi_0))=1$,
hence by Remark \ref{noinvolution}, the ideal of reducibility $I_{\rm re} \subset R_{\Sigma}$ is principal. \end{rem}

From now on assume that $\phi$ is unramified or that we are in the situation of Theorem 4.4 of \cite{BergerKlosin09}. Let $\bfT_{\Sigma}$ denote the
Hecke algebra defined in section 4 of \cite{BergerKlosin09}, except we do not restrict to the ordinary part.  Conjecture 5 of
\cite{B09} asserted
that the Galois representation $\rho_{\pi}$ attached to an automorphic representation $\pi$ over $F$ is crystalline if $\pi$ is unramified at $p$. This has now been proven in many cases by A. Jorza \cite{Jorza10}.  When it is satisfied we obtain by universality a canonical map $\psi: R_{\Sigma} \twoheadrightarrow \bfT_{\Sigma}$ and the set $\Pi$ in section \ref{RTtheorems} can be
identified with the set $\Pi_{\Sigma}$ from section 4.2 of \cite{BergerKlosin09}. By Theorem 14 of \cite{B09} condition (2) of Theorem \ref{cons1} is satisfied
with $m$ as in Conjecture \ref{BK1}. Hence gathering all this, we can apply Theorem \ref{cons1} (using Remark \ref{symmetry2} instead of the existence
of $\tau$) to deduce the following modularity result.

\begin{thm} \label{imquad} The map  $\psi: R_{\Sigma} \twoheadrightarrow \bfT_{\Sigma}$ is an isomorphism. \end{thm}

From this one easily has the following modularity theorem.
\begin{thm} \label{modimquad} Let $F$, $p$ and $\Sigma$ be as above. Let $\phi$ be an unramified Hecke character of infinity type $z^2$ and let $\chi_0=\ov{\phi_{\fp} \epsilon}$. Assume $\chi_0$ is $\Sigma$-minimal and that Conjecture \ref{BK1} holds $\phi$. Let $\rho: G_{\Sigma} \rightarrow \GL_2(E)$ be an irreducible continuous Galois representation and suppose that $\ov{\rho}^{\rm ss} \cong 1 \oplus \chi_0$. If $\rho$ is crystalline at the primes of $F$ lying over $p$ then (a twist of) $\rho$ is modular. \end{thm}

\section{4-dimensional Galois representations of $\bfQ$ - Yoshida lifts} \label{Examples2}

In this section we apply our methods to study the deformation ring of a 4-dimensional residually reducible Galois representation of $G_{\bfQ}$.
\subsection{Setup} Let $S_n(N)$ denote the space of (elliptic) cusp forms of weight $n$ and level $N$. Assume that $p>k\geq 4$ is even and that $N$ is
a square-free positive integer with $p \nmid N$. We will also assume that all primes $l \mid N$ satisfy $l \not\equiv 1$ (mod $p$). Let $f\in S_2(N)$ and $g \in S_k(N)$ be two eigenforms whose residual (mod $p$) Galois representations are absolutely
irreducible and mutually non-isomorphic. For a positive integer $n$ write $S^S_n(N)$ for the space of Siegel modular forms $\phi$ which are cuspidal
and satisfy $$\det(CZ+D)^{-n}\phi((AZ+B)(CZ+D)^{-1})=\phi(Z) \quad \textup{for} \quad \bmat A&B \\ C&D \emat \in \Sp_4(\bfZ); C\equiv 0 \pmod{N}.$$
Here $Z$ is in the Siegel upper-half space.

\begin{thm} [Yoshida]\label{Yoshidathm} There exists a $\bfC$-linear map $$Y:S_2(N)\otimes S_k(N) \rightarrow S^S_{k/2+1}(N)$$ such that $$L_{\rm
spin}(s, Y(f\otimes g)) = L(s-k/2+1, f)L(s,g)$$ up to the Euler factors at the primes dividing $N$. In particular the lift $Y(f\otimes g)$ is
a Hecke eigenform for primes away from $N$. \end{thm}

Let $\Sigma$ denote the finite set of finite places of $\bfQ$ consisting of $p$ and the primes dividing $N$. For a Siegel cuspidal eigenform $\phi$
(away from $\Sigma$) denote by $\rho_{\phi}: G_{\Sigma} \rightarrow \GL_4(E)$ the Galois representation attached to $\phi$ by Weissauer \cite{Weissauer05} and Laumon \cite{Laumon05}. The representations are crystalline at $p$ by \cite{ChaiFaltings90} Th\'eor\`eme VI.6.2.
It follows from Theorem
\ref{Yoshidathm} that $$\rho_{Y(f\otimes g)} \cong \rho_f (k/2-1) \oplus \rho_g,$$ where $\rho_f$ and $\rho_g$ denote the Galois representations
attached to $f$ and $g$ by Eichler, Shimura and Deligne. Note that because the determinants of the two two-dimensional summands match, the image of
$\rho_{Y(f\otimes g)}$ is contained (possibly after conjugating) in $\GSp_4(\Oo)$ and not just in $\GL_4(\Oo)$. Let $S^{\rm nY}$ denote the
orthogonal complement (under the standard Petersson inner product on $S^S_{k/2+1}(N)$) of the image of the map $Y$ and let $S^{f,g} \subset S^{\rm
nY}$ denote the subspace spanned by eigenforms $\phi$ whose Galois representation satisfy the following two conditions:
\begin{itemize}
 \item $\rho_{\phi}$ is irreducible;
\item The semisimplification of the reduction mod $\varpi$ (with respect to some lattice in $E^4$) of $\rho_{\phi}$ is isomorphic to
    $\ov{\rho}_f
    (k/2-1) \oplus \ov{\rho}_g$.
\end{itemize}
Let $\bfT^{S}$ denote the $\Oo$-subalgebra of $\End_{\Oo}(S^S_{k/2+1}(N))$ generated by the local Hecke algebras away from $\Sigma$, and let
$\bfT_{\Sigma}=\bfT^{f,g}$ be the image of $\bfT^{S}$ inside $\End_{\Oo}(S^{f,g})$. Then (if non-zero) $\bfT^{f,g}$ is a local, complete Noetherian $\Oo$-algebra
with residue field $\bfF$ which is finitely generated as a module over $\Oo$. Let $\Ann(Y(f\otimes g))\subset \bfT^S$ denote the annihilator of
$Y(f\otimes g)$. It is a prime ideal and one has $\bfT^S/\Ann(Y(f\otimes g)) \cong \Oo$. Let $I_{f,g}=\psi(\Ann(Y(f\otimes g)))$, where $\psi: \bfT^S
\twoheadrightarrow \bfT^{f,g}$ is the projection map. It is an ideal.
\begin{conj} \label{conj2}  Suppose $m=\val_{\varpi}(L^{N, \rm alg}(1+k/2, f\times g))>0$. Then $$\# \bfT^{f,g}/ I_{f,g} \geq \#\Oo/\varpi^m.$$ Here
$L^{N, \rm alg}(1+k/2, f\times g)$ denotes appropriately normalized special value of the convolution $L$-function of $f$ and $g$. \end{conj}
 In a recent preprint Agarwal and the second author have proved this conjecture in many cases (cf. \cite{AgarwalKlosin10preprint}, Theorem 6.5 and Corollary 6.10) under some additional
 assumptions (among them that $f$ and $g$ are ordinary). See also \cite{BochererDummiganSchulzePillotpreprint} for a similar result. As a consequence of this conjecture we get that $\bfT^{f,g} \neq
0$ whenever the $L$-value is not a unit. Also, the conjecture implies that the space $S^{f,g} \neq 0$. Let $F \in S^{f,g} \neq 0$ be an eigenform.
Then its Galois representation $\rho_F : G_{\Sigma} \rightarrow \GL_4(E)$ is irreducible, but the semi-simplification of its reduction mod $\varpi$
has the form $\ov{\rho}_F^{\rm ss} \cong \ov{\rho}_f(k/2-1) \oplus \ov{\rho}_g.$ Using Proposition \ref{genRibet} we fix a lattice $\mL
\subset E^4$ such that with respect to that lattice $\ov{\rho}_F$ is non-semi-simple and has the form $$\ov{\rho}_F = \bmat \ov{\rho}_f(k/2-1) &*\\ 0&\ov{\rho}_g\emat.$$  Set $$\rho_0:= \ov{\rho}_F.$$

\subsection{Assumption \ref{mainass}} In what follows we impose Assumption \ref{mainass}. Let us briefly discuss some sufficient conditions under
which Assumption \ref{mainass} is satisfied. The Selmer group in Part (1) is equal to $H^1_{f}(\bfQ, \Hom_{\Oo}(\tilde{\rho}_2,
\tilde{\rho}_1))\otimes \bfF$ as long as we assume the conditions of Lemma \ref{Tamagawa}. The condition on the Tamagawa factor is satisfied if $W^{I_v}$ is divisible, where $W=\Hom_{\Oo}(\tilde{\rho}_2,
\tilde{\rho}_1))\otimes E/\Oo$. This is proven in \cite{BochererDummiganSchulzePillotpreprint} Lemma 3.2(i) under the additional assumption that there does not exist a newform $h\in S_2(N)$, $h \neq f$ which is congruent (away from $\Sigma$) to $f$ (mod
$\varpi$) and similarly  there does not exist a newform $h\in S_k(N)$, $h \neq g$ which is congruent (away from $\Sigma$) to $g$ (mod $\varpi$). In what follows we assume that the local $L$-factors in Lemma \ref{Tamagawa} are $p$-adic units, and hence a necessary and sufficient condition for the Assumption
\ref{mainass}(1) to be satisfied is that the Bloch-Kato Selmer group be cyclic (which, assuming the relation to an $L$-value predicted by the Bloch-Kato conjecture, is guaranteed
for example when $\val_{\varpi}(L^{N, \rm alg}(1+k/2, f\times g))=1$).

On the other hand one can also formulate some sufficient conditions under which Assumption \ref{mainass}(2) is satisfied. We will only discuss the case of $\rho_2=\ov{\rho}_g$, the case of $\rho_1$ being similar. Suppose that $\rho: G_{\Sigma} \rightarrow \GL_2(\Oo)$ is another crystalline lift of $\rho_2$. In particular $\rho$ is semi-stable at $p$, hence
the Fontaine-Mazur conjecture predicts that it should be modular. This conjecture is true in many cases. In particular it is true when $\rho$ is unramified outside finitely primes, ramified at $p$ and (short) crystalline, with $\ov{\rho}|_{G_{\bfQ(p)}}$ absolutely irreducible and modular (here $\bfQ(p) = \bfQ(\sqrt{(-1)^{(p-1)/2}p})$) by a Theorem of Diamond, Flach and Guo (\cite{DiamondFlachGuo04}, Theorem 0.3). In our case $\rho$ is ramified at $p$ because $\det \ov{\rho}$ is, so if we assume in addition that $\ov{\rho}_g|_{G_{\bfQ(p)}}$ is absolutely irreducible, we can conclude that there exists a modular form $h$ such that $\rho \cong \rho_h$. Since we assume (in accordance with Assumption \ref{mainass} - see discussion following that assumption) that $\Sigma$ does not contain any primes congruent to 1 mod $p$, we have $H^1_{\Sigma}(F, \ad^0 \tilde{\rho}_i\otimes E/\Oo) = H^1_{\Sigma}(F, \ad \tilde{\rho}_i\otimes E/\Oo)$ as explained in section \ref{Main assumptions}. Hence we must have $\det \rho_h = \det \rho_g$, so $h$ is necessarily of weight $k$. 
Since our deformations are unramified outside $\Sigma$ (and crystalline at $p$), the level of the form $h$ can only be divisible by the primes dividing $N$.
In this case Assumption \ref{mainass}(2) is equivalent to an assertion
that there does not exist a newform $h\in S_2(N^2)$, $h \neq f$ which is congruent (away from $\Sigma$) to $f$ (mod
$\varpi$) and similarly  there does not exist a newform $h\in S_k(N^2)$, $h \neq g$ which is congruent (away from $\Sigma$) to $g$ (mod $\varpi$). Indeed, it follows from a result of Livne (\cite{Livne89}, Theorem 0.2) that under our assumptions concerning the primes in $\Sigma$, the form $f$ (resp. $g$) cannot be congruent to a form of level divisible by $l^3$ for a prime $l\mid N$ (note that $N$ is square-free by assumption). Alternatively one can use Theorem 1.5 in \cite{Jarvis99} which works for all totally real fields.
 So, Assumption \ref{mainass}(2) 
follows from just a slight strengthening of the congruence conditions already imposed to satisfy Assumption \ref{mainass}(1).

Alternatively, the Selmer group $H^1_{\Sigma}(\bfQ, \ad \rho_i)$, $i=1,2$ could be related to a symmetric-square $L$-value using the Bloch-Kato conjecture and Lemma \ref{Tamagawa} together with Remark \ref{tamagawaremark}. For the divisibility of $(W^*)^{I_v}$ for $W=\ad^0{\rho}_f \otimes E/\Oo$ we can argue as follows, as explained to one of us by Neil Dummigan: Assume again that $f$ is not congruent (away from $\Sigma$) to another newform modulo $\varpi$. Then with respect to some basis $x,y$, both $\rho_f$ and
$\overline{\rho}_f$ send a generator of the $p$-part of the tame inertia group at $v$ to the matrix $\begin{pmatrix}
  1&1\\0&1\end{pmatrix}$. It follows that the
$I_v$-fixed parts of both ${\rm Sym}^2 \rho_f$ and ${\rm Sym}^2 \ov{\rho}_f$ are two-dimensional,
spanned by $x^2$ and $xy-yx$. Hence the $I_v$-fixed part
of ${\rm Sym}^2 \rho_f \otimes E/\Oo$ is divisible.
We observe that ${\rm Sym}^2 \rho_f$ differs from $W^*$ just by a Tate twist.

\subsection{Deformations}
From now on we assume that Assumptions \ref{mainass}(1) and (2) hold, $\Sigma=\{l \mid N\}\cup \{p\}$.
  Consider an anti-automorphism $\tau: G_{\Sigma} \rightarrow G_{\Sigma}$ given by $\tau(g)=\epsilon(g)^{k-1}g^{-1}$ (see Example
  \ref{exampleaa}(3)). Note that $\tr \rho_i \circ \tau = \tr \rho_i$ for $i=1,2$. By Remark 1 of \cite{Weissauer05} we also know that for any Siegel modular form $\phi$ of parallel weight $k/2+1$, the Galois representation $\rho_{\phi}$ (in particular, also $\rho_0$) is essentially self-dual with respect to $\tau$ as defined above, i.e. that $\rho_{\phi}^* \cong \rho_{\phi} \epsilon^{1-k}$.

We study
deformations $\rho$ of $\rho_0$ such that  \begin{itemize}  \item
$\rho$ is crystalline at $p$; \item $\tr \rho \circ \tau = \tr \rho$.
\end{itemize}
This deformation problem is represented by a universal couple $(R_{\Sigma}, \rho_{\Sigma})$. By Proposition \ref{princi2} the ideal of reducibility
$I_{\rm re}$ of $R_{\Sigma}$ is principal. Moreover since $R_{\Sigma}$ is generated by traces (Proposition \ref{traces}), we get an $\Oo$-algebra
surjection $\phi:R_{\Sigma} \twoheadrightarrow \bfT^{f,g}$. (Note that even though the Hecke operators are involved in all the coefficients of the
characteristic polynomial of the Frobenius elements, all of them can be expressed by the trace.) The ($\varpi$-part of the) Bloch-Kato conjecture (together with Lemma \ref{Tamagawa} - see the discussion above)
predicts that \be \label{MC} \# H^1_{\Sigma}(\bfQ, \Hom_{\Oo}(\rho_g, \rho_f(k/2-1))) \leq \#\Oo/\varpi^m\ee with $m$ as above. At the moment this
conjecture is beyond our reach. Moreover, it is not clear that the periods used to define the algebraic $L$-value involved in the Main Conjecture and
the one defining $L^{N, \rm alg}$ above coincide (something we have assumed when writing (\ref{MC})). If we assume (\ref{MC}) then Proposition
\ref{Ovarpis} implies that $R_{\Sigma}/I_{\rm re} \cong \Oo/\varpi^s$ for $s\leq m$. So the induced map $R_{\Sigma}/I_{\rm re} \twoheadrightarrow
\bfT^{f,g}/\phi(I_{\rm re}) = \bfT^{f,g}/I_{f,g}$ is an isomorphism. Thus by Theorem \ref{cons1}, we get that $\phi$ is an isomorphism. In particular
we have proved the following theorem:
\begin{thm}\label{modsiegel} Let $f$, $g$ and $\Psi$ be as above and assume that Assumption \ref{mainass} is satisfied and that equation (\ref{MC})
as well as Conjecture \ref{conj2} hold.
%We assume that the local Euler factors $P_v(V^*,1)$ from Lemma \ref{Tamagawa} are $p$-adic units for $V=\Hom(\tilde \rho_i,\tilde \rho_j)$ for $i,j \in \{1,2\}$ and $v \in \Sigma$, $v \nmid p$. Further assume that $f$ and $g$ are not congruent (away from $\Sigma$) to other newforms modulo $\varpi$.
Let $\rho: G_{\Sigma} \rightarrow \GL_4(E)$ be an irreducible Galois representation and suppose that
$$\ov{\rho}^{\rm ss} = \ov{\rho}_f (k/2-1) \oplus \ov{\rho}_g.$$
 Moreover assume  that $\rho$ is crystalline. Then there exists $F' \in S^S_{k/2+1}(N)$ such that $$\rho \cong
 \rho_{F'},$$ i.e., $\rho$ is modular. \end{thm}

\bibliographystyle{amsalpha}
\bibliography{standard2}

\end{document}

%% file: makros.tex
\def\a{\alpha}

\newcommand{\mA}{\mathcal{A}}
\newcommand{\mB}{\mathcal{B}}

\newcommand{\mE}{\mathcal{E}}

\newcommand{\mL}{\mathcal{L}}

\newcommand{\mS}{\mathcal{S}}
\newcommand{\mT}{\mathcal{T}}

\newcommand{\fm}{\mathfrak{m}}

\newcommand{\fp}{\mathfrak{p}}
\newcommand{\fP}{\mathfrak{P}}
\newcommand{\fq}{\mathfrak{q}}

\newcommand{\bfC}{\mathbf{C}}

\newcommand{\bfF}{\mathbf{F}}
\newcommand{\bfG}{\mathbf{G}}

\newcommand{\bfQ}{\mathbf{Q}}
\newcommand{\bfR}{\mathbf{R}}

\newcommand{\bfT}{\mathbf{T}}

\newcommand{\bfZ}{\mathbf{Z}}

\newcommand{\Oo}{\mathcal{O}}

\newcommand{\tuint}{\textup{int}}

\newcommand{\ture}{\textup{re}}

\newcommand{\ov}{\overline}

\newcommand{\be}{\begin{equation}}
\newcommand{\ee}{\end{equation}}
\newcommand{\bes}{\begin{equation*}}
\newcommand{\ees}{\end{equation*}}

\newcommand{\bs}{\begin{split}}
\newcommand{\es}{\end{split}}
\newcommand{\bss}{\begin{split*}}
\newcommand{\ess}{\end{split*}}

\newcommand{\bmat}{\left[ \begin{matrix}}
\newcommand{\emat}{\end{matrix} \right]}
\newcommand{\bsmat}{\left[ \begin{smallmatrix}}
\newcommand{\esmat}{\end{smallmatrix} \right]}

\newcommand{\bml}{\begin{multline}}
\newcommand{\eml}{\end{multline}}
\newcommand{\bmls}{\begin{multline*}}
\newcommand{\emls}{\end{multline*}}

\DeclareMathOperator{\ad}{ad}
\DeclareMathOperator{\Ann}{Ann}

\DeclareMathOperator{\Cl}{Cl}
\DeclareMathOperator{\coker}{coker}

\DeclareMathOperator{\End}{End}

\DeclareMathOperator{\Frob}{Frob}
\DeclareMathOperator{\Gal}{Gal}
\DeclareMathOperator{\GL}{GL}
\DeclareMathOperator{\GSp}{GSp}

\DeclareMathOperator{\Hom}{Hom}

\DeclareMathOperator{\Sp}{Sp}

\DeclareMathOperator{\Tam}{Tam}

\DeclareMathOperator{\val}{val}

\newcommand{\hs}{\hspace{2pt}}
\newcommand{\hf}{\hspace{5pt}}

\newcommand{\tr}{\textup{tr}\hspace{2pt}}

\newcommand{\invlim}{\mathop{\varprojlim}\limits}
\newcommand{\dirlim}{\mathop{\varinjlim}\limits}

%% file: settings.tex
\usepackage[all]{xy}
\usepackage{amsthm}
\usepackage{amssymb, amsfonts}
\SelectTips{cm}{10}\UseTips
\theoremstyle{plain}
\newtheorem{thm}{Theorem}
\newtheorem{prop}[thm]{Proposition}
\newtheorem{cor}[thm]{Corollary}
\newtheorem{lemma}[thm]{Lemma}
\newtheorem{conj}[thm]{Conjecture}

\theoremstyle{definition}
\newtheorem{definition}[thm]{Definition}

\newtheorem{example}[thm]{Example}
\newtheorem{rem}[thm]{Remark} 
\newtheorem{assumption}[thm]{Assumption}

\numberwithin{thm}{section}
\numberwithin{equation}{section}